\documentclass[12pt]{article}  

\usepackage[fontsize=12pt]{fontsize}
\usepackage{amsmath}
\usepackage{amsthm}
\usepackage{amssymb}
\usepackage{graphicx}
\usepackage{geometry}
\usepackage{hyperref}
\usepackage{subcaption}
\usepackage{booktabs}
\usepackage{natbib}
\usepackage{bm}
\usepackage{listings}
\usepackage{algorithm}
\usepackage{algorithmic}
\usepackage{colortbl}
\usepackage{multirow}
\usepackage{enumerate}

\usepackage{setspace}
\theoremstyle{plain}
\newtheorem{theorem}{Theorem}%[section]

\newtheorem{consequence}[theorem]{Consequence}
\newtheorem{lemma}[theorem]{Lemma}
\newtheorem{proposition}[theorem]{Proposition}
\theoremstyle{definition}
\newtheorem{remark}{Remark}
\newtheorem{definition}[section]{Definition}
\title{Optimal testing in linear models with implications for causal inference}
\author{Razvan G Romanescu \\ \texttt{Razvan.Romanescu@umanitoba.ca } \\
College of Community and Global Health, Rady Faculty of Health Sciences,\\ University of Manitoba, 
Winnipeg MB, Canada%  R3E 0T6
}

\begin{document}

\maketitle

\abstract{Despite its widespread use in practice, multiple regression suffers from two important limitations due to ignoring the structure between predictors. First, reported coefficients are, in fact, direct effects in a causal context. These are often less meaningful than total effects when accounting for a factor’s contribution, however, estimating the latter typically requires a causal diagram to be specified among predictors. The second limitation is multicollinearity, which causes a number of issues for inference, in particular loss of power. While any orthogonal transformation resolves multicollinearity, only the Gram--Schmidt decomposition resolves both limitations by transforming direct effects into total effects, in a strictly--ordered  temporal causal diagram. This directed acyclic graph (DAG) is of notable importance in practice, where causal precedence often comes from temporal separation. 
Identifying total effects with vector directions leads to exact distributions for the indirect effect and fraction mediated for each predictor, with equal or better powered tests compared to existing mediation software. We establish conditions under which testing for total effects is more powerful than coefficient tests in regression, which is a compelling reason for the use of temporal DAGs in general practice. 
The paper further develops the theory of indirect effects decomposition and quantifies how effects originating from a variable or set propagate via arbitrary downstream mediators or sets. This is applied to lag detection in causal discovery, by testing the flow of indirect effects between variables at specific lags in a feedback loop. Unlike standard time series methods, this approach is robust to lack of stationarity.}\\
%Sample size calculations are much simplified using this approach, and we apply it to calculate the sample size needed for a study on the policy value of offering incentive-based HIV testing in Malawi. Beyond the contribution to mediation analysis, we argue that power gains under this approach may be significant enough to justify switching from regression to a } \\


% \boxedtext{
% \begin{itemize}
% \item Key boxed text here.
% \item Key boxed text here.
% \item Key boxed text here.
% \end{itemize}}

\section{Introduction}
A persistent problem in multiple regression (MR) is that correlated predictors lead to loss of power and/or stability issues. In an extreme case, including perfectly correlated predictors leads to a model that is over-identified and cannot be fitted. Even if features are highly, but not perfectly correlated, multicollinearity might make coefficient variances large and point estimates highly sensitive to the paricular values in the design matrix, making  the fit unstable and replication difficult. The amount of multicollinearity is sometimes measured via variance inflation factors (VIFs). Parameters with high VIF are deemed to significantly increase multicollinearity of the model and are often discarded. This not only results in loss of information, but may also not completely eliminate multicollinearity among the remaining predictors.%, instead only alleviating the problem.  %However, this may be  undesirable as each predictor contains unique information, which is lost. 

%Theoretical discussions in multiple regression  have so far focused on the properties of the OLS estimator, namely that it is BLUE and  BUE (see \cite{hansen, potscher, portnoy}). This treatment, however, remains within the space of the original regressors and does not address the practical problem of multicollinearity.
%While perhaps less of a problem in estimation, correlation between predictors can result in significant loss of power when performing hypothesis tests. 
%Derivative models that attempt do deal with this issue, such as ridge regression, have already been. 

In this paper, we focus on power for testing features in linear models, and start our investigation from the question of whether a uniformly most powerful (UMP) test exists for testing the coefficient of a predictor of interest in multiple regression. 
%According to the Lehman--Scheff\'e theorem, any unbiased estimate that depends on the data only via the sufficient statistics is the unique uniformly minimum variance unbiased estimator (UMVUE). A test based on such a quantity would necessarily have better properties compared to a test based on any other unbiased estimate, however, it does not directly follow that this test is UMP. The theory for finding the most powerful test --- when it exists --- is based on different mechanics that do not call for an unbiased estimator at all. In fact, a decision rule used to test hypotheses about a parameter need not be based on an estimate of that parameter. Instead, finding the uniformly most powerful test for a parameter of interest in the presence of nuisance parameters relies on the notion of Neyman--structure of tests with respect to the sufficient statistic, and that of ``unbiased'' tests at level $\alpha$, which has a different meaning related to the distribution of power over the parameter space.
We show that a t--test for coefficients based on OLS ends up being the UMP unbiased test  in the multiple regression model. Though similar to  estimation theory including properties such as BLUE, BUE (see \cite{hansen, potscher, portnoy}) and the uniformly minimum variance unbiased estimator (UMVUE), the UMP property has not been shown before, and is arrived at via a distinct path, which does not require unbiased estimators. However, it quickly becomes apparent that the UMP framework is unsatisfactory for power, for two reasons: 
1) just because a test  is UMPU under one model does not stop one from finding a different, related model that offers a more powerful test for the same feature. For example, penalizing a regression model by ridge regression  has been  shown to have improved power for testing feature coefficients,  compared to the original model (\cite{halawa}). 
And 2) standard coefficient tests based on  OLS estimates are  still  plagued by multicollinearity, despite being UMPU, and thus may be severely underpowered compared to tests based on a new, orthogonal feature basis. 
A complete resolution of the issue of multicollinearity is afforded by any orthogonal rotation of the design matrix. Among these, the change of basis given by the Gram--Schmidt (GS) decomposition  is arguably unique in keeping a meaningful interpretation of the new features. 
The GS algorithm  traces its origins to Laplace (\cite{laplace}) and is one way to obtain the QR factorization. Gram-Schmidt  regression has been explicitly introduced as such half a century ago (\cite{farebrother}), but remains underused in the statistical sciences, despite its usefulness having been recognized in other fields (see, e.g.,  \cite{clyde, klein, forina}).
% In Section  \ref{size} we formally compare the GS and multiple regression models in terms of power, and show that the implications for study design are tangible and significant. While the power gains are impressive, interpretation may be key to wider adoption, and in Section \ref{multicoll} we discuss more in depth how to interpret GS results and effect size estimates in the context of multicollinearity and when this model might be more appropriate to use in place of multiple regression. 
A hurdle in the adoption of GS regression in statistical practice has been, perhaps, the difficulty in interpreting the transformed features. 

Recently, a connection has been made with mediation analysis, and the GS transformed features were shown to be consistent with a temporal causal diagram, i.e., one defined by the time separation between cause and effect variables in which the direction of causation matches the order in which variables are orthogonalized (\cite{cross}). 
The interpretation of GS along a DAG constitutes a novel, vector--based approach to mediation analysis, 
 which  greatly simplifies  calculations  because measuring the effects directly on an orthogonal basis removes the need for path counting.  This paper comprehensively develops estimation and testing theory for causal effects in this class of DAGs, by using the GS--transformed predictors as the innovation terms that can be causally acted upon for each variable. We spell out the exact tests for direct, indirect, and total effects of a feature, as well as derive the distribution for the fraction of total effect that is mediated.
The resulting closed--form distributions allow one to perform power calculations for mediation studies much more easily using this approach. We illustrate this on a policy study of HIV testing in Malawi.  

Beyond helping to formalize the above theory, this paper makes some uniquely novel contributions:  1) it defines mediation through any downstream predictor atomically. This opens the door to testing at the more granular level whether variable $A$ is mediated by $B$, and we show how individual mediation instances can be aggregated to demonstrate causation at a lag in a feedback loop. 2) For the assumed temporal DAG % causal precedence of variables, %we derive a diagnostic metric similar to, but arguably superior to the VIF, which indicates when hypothesis tests for total effects will be more powerful compared  to OLS coefficient tests in multiple regression.
we derive the conditions under which testing for  total effects will be more powerful compared  to testing for the OLS coefficients of the respective variables. This can be determined ahead of time from pilot data and can determine if a mediation analysis should be preferred for the new study instead of the association analysis by regression.
%a diagnostic metric similar to, but arguably superior to the VIF, which indicates when hypothesis tests for total effects will be more powerful compared  to OLS coefficient tests in multiple regression.

%Part I 
\section{Conditionally best tests in regression} \label{methods}
We begin by establishing some optimality results for testing in regression.
 Building UMPU tests is well established in inference theory, especially for distributions in the single parameter exponential family. The existence of UMP tests in this case is based on the Neyman-Pearson Lemma, and tests can be built by  writing the likelihood ratio as a monotone function of the sufficient statistic. While this approach does not generalize directly to multi-parameter families, UMPU tests can be constructed for a parameter of interest by conditioning on the sufficient statistics for the other (nuisance) parameters.

\subsection{Related work} %This is the problem we consider here, namely making inference about the coefficient of a single fixed effect in multiple regression, while treating the other coefficients as nuisance. 
A UMP invariant (UMPI) test for the directional testing of a subset of coefficients being jointly zero, assuming knowledge of the coefficients' signs, has been constructed by \cite{king}. The invariance condition is a somewhat strong  assumption, and this test does not attain the envelope of power, even though it is shown to perform reasonably well in simulations. A UMP test for the variance parameter in regression was derived by \cite{zhang} under a more lenient assumption than unbiasedness. 
The problem of efficient testing in parametric models in the large sample limit has been solved for a general distribution by \cite{choi}, by  using the notions of asymptotically uniformly most powerful (AUMP), and effective scores. However, these are advanced theoretical concepts based on local asymptotic normality, and no simple solution has been derived for multivariate regression, which is an important case in applied statistics. The treatment we consider here is exact as opposed to asymptotic, and, as such works for small samples as well as large. Importantly, we wish to obtain the test in closed form, and establish its link to familiar test statistics from regression analysis.

\subsection{Regression on an orthonormal set of predictors} %\label{methods}

We are concerned with finding an optimal one-sided test for a coefficient in the multiple regression model. It is easier to do this first when features are orthonormal:

\begin{lemma}  \label{thm1}
 Suppose we observe data vector $\mathbf{Y}$  from the multiple regression model $\bm{Y}=\beta_1 \mathbf{x}_{1}+\beta_2 \mathbf{x}_{2}+...+\beta_p \mathbf{x}_{p}+ \bm{\epsilon}$, where $\bm{\epsilon} \sim N(0,\sigma^2 I),$ and  $\mathbf{x_1,x_2,...,x_p}$ are fixed covariates, for $p<n$. Assume further that $\mathbf{x_1,x_2,...,x_p}$ are orthonormal, and all parameters ($\beta_1, \beta_2,...,\beta_p$ and $\sigma^2$) are unknown.  The test $\phi$ defined as $\phi=I\{ V\geq t_{n-p,1-\alpha} \}$,
%\begin{equation} \label{phi}
%\phi=\begin{cases}
%0, \quad  \text{if $V<t_{n-p,1-\alpha}$} \\
%1, \quad  \text{if $V\geq t_{n-p,1-\alpha}$,}
%\end{cases}
%\end{equation}
where $V=\frac{\sqrt{n-p}\, \mathbf{x}_p^\intercal \mathbf{Y}}{\sqrt{\mathbf{Y}^\intercal \mathbf{Y} -  \sum_{i=1}^p{(\mathbf{x}_i^\intercal \mathbf{Y})^2} }} \sim t_{n-p}$
is UMPU for testing $H_0: \beta_p \leq 0$ vs $H_1: \beta_p > 0$.
\end{lemma} 

 The proof (see Appendix) generalizes Example 6.9.11 from \cite{bhattacharya}, which establishes the result in the more limited case of testing for the slope in a simple regression model, in which the intercept and error variance are unknown.
We observe that test statistic $V$ is identical to the test of coefficient $\beta_p$ being significantly different from zero. This t--test is standard output when fitting a multiple regression in most statistical software packages. This identification can be seen by writing  $V=\frac{ \hat{\beta}_p }{\sqrt{SSE /(n-p)}}=\frac{ \hat{\beta}_p }{\text{s.e.}( \hat{\beta}_p)}$, which is the Student$-$t test statistic for coefficient $\beta_p$ . Here we have used the fact that $\text{s.e.}( \hat{\beta}_p) = \sqrt{s^2 (X^\intercal X)^{-1}_{pp}}=\sqrt{s^2 I_{pp}} = s$.
The degrees of freedom are also the same: since we have considered the intercept to be one of the predictors, we would have $p-1$ ``predictor variables'' in the standard textbook formulation of the model, so the degrees of freedom associated with the sum of squares $SSE$ would be $n-p$, the same as in the previous  lemma.

%\vspace*{-3em}

\subsection{Transforming the predictor set via Gram--Schmidt}\label{GS}
%The next question to ask  is whether the previous result generalizes  for correlated predictors.
A key property in Lemma \ref{thm1}  was that the estimate of the coefficient of interest did not depend on the other features; this will not be the case under an arbitrary correlation structure. %The associated hypothesis test depended on 
However, the model hyperplane, i.e., the span of all features, remains the same under a change of basis.
We use  the Gram--Schmidt algorithm to obtain an orthogonal basis, which reduces the conditions to that of  Lemma \ref{thm1}.  
The specific implementation we use to orthogonalize a set of $p$ features $\bm{m}_1,..., \bm{m}_p$ is summarized in Algorithm \ref{algo}.  %We may use the following notation to denote the new basis, because it keeps track of the order of orthogonalization:

\begin{algorithm} 
\caption{(A variant of) the Gram--Schmidt algorithm to orthogonalize a feature set around the first direction.}
\begin{algorithmic}[1]
    \STATE Fix the first basis vector to $\bm{x}_1  = \frac{\bm{m}_1}{||\bm{m}_1||}$, where  $\bm{m}_1$ is the feature of interest
        \FOR{$k \gets 2$ \textbf{to} $p$}
            \STATE  Regress the $\bm{m}_k$-th predictor on the basis vectors obtained so far, i.e.,  $\bm{m}_k = \alpha_{k,1}\bm{x}_1 + ...+ \alpha_{k,k-1}\bm{x}_{k-1}+ \bm{r}_k$ 
            \STATE Set the next basis vector, $\bm{x}_k$, as the component of $\bm{m}_k$ orthogonal to $\bm{x}_1,..., \bm{x}_{k-1}$, i.e.,   $\bm{x}_k  = \frac{\bm{\hat{r}}_k}{||\bm{\hat{r}}_k||}$
            \STATE  Compute the $k$-th column of matrix $Q$ as $(\hat{\alpha}_{k,1}, \hat{\alpha}_{k,2},..., \hat{\alpha}_{k,k-1}, ||\bm{\hat{r}}_k||, 0,..,0)^\intercal$
        \ENDFOR
\end{algorithmic}
\label{algo}
\end{algorithm}

Essentially, Gram--Schmidt solves for an upper triangular matrix $Q$ which transforms the original set of features into an orthogonal set, such that 
$$ (\bm{m}_1, \bm{m}_2, \bm{m}_3,..., \bm{m}_p)= (\bm{m}_1^\perp, \bm{m}_2^{\perp \,1}, \bm{m}_3^{\perp (1,2)},..., \bm{m}_p^{\perp (1,...,p-1)}) Q=XQ,$$
where we have used the notation $ (\bm{m}_1^\perp, \bm{m}_2^{\perp \,1}, \bm{m}_3^{\perp (1,2)},..., \bm{m}_p^{\perp (1,...,p-1)}) \triangleq ( \bm{x}_1,  \bm{x}_2,..., \bm{x}_p).$
%For example, the closed form solution for a basis with up to three features is given in the Appendix. 
%From the point of view of interpretation, it is important to note that the meaning of the original predictors is partly preserved, as opposed to other algorithms (such as principal components) where the new directions may not be meaningfully related to the original features. In this case, 
Each  basis vector of the new predictor set represents an ``innovation", or remainder, that could not be explained by the previous basis vectors. As a concrete example, if we were regressing some overall health score on age first, then smoking status, the coefficients of the new terms $\text{age}^\perp$ and  $\text{smoking}^{\perp \text{age}}$ would capture, respectively: (i) the unconditional marginal association with age, including direct and indirect effects --- this would be identical to a marginal regression on age alone; and (ii) any residual association between smoking and health, over and above the effects of age. 
%It is obvious that the interpretation of all the new terms except for the first one is dependent on the sequence of orthogonalization. More on the importance of ordering will be discussed in Section \ref{multicoll}.
%The computational advantages of this transformation in the context of regression have been documented  in particular in Mathematical Chemistry, where it found application  in quantitative structure–activity relationship (QSAR) models used to predict the behavior of chemical compounds. 
Some of the documented benefits of this transformation in the context of regression   include the stability of coefficient estimates when new predictors are added to the model, as well as circumventing the problem of multicollinearity (\cite{randic, randic.book}, and others). %(pp. 311–324) 

\subsection{Multiple regression on correlated predictors}\label{multiple}
Equipped with the ability to find an equivalent, orthogonal basis for predictors, we can now prove the more general result for correlated independent variables.

 \begin{theorem}\label{thm2}
A one--sided coefficient t--test based on the OLS estimate in multiple regression is UMPU. 
%\begin{align*}
%\bm{Y} &=\alpha_1 \mathbf{m}_{1}+\alpha_2 \mathbf{m}_{2}+...+\alpha_p \mathbf{m}_{p}+ \bm{\epsilon}, \text{and}  \tag{A}\\
%\bm{Y} &=\beta_1 \mathbf{x}_{1}+\beta_2 \mathbf{x}_{2}+...+\beta_p \mathbf{x}_{p}+ \bm{\epsilon} \tag{B}
%\end{align*}
\end{theorem}  

The proof uses the notation $M_{i*}$  to denote the $i$-th row of matrix $M$, seen as a $1\times p$ matrix; similarly for columns.  As well, $\bm{q}_i^\intercal = Q_{i*}^{-1}$, and we shall continue using the vector notation when shorter. 
This ends the first part of the paper, where we have established that coefficient t-tests in regression are UMP. To the author's knowledge, this has not been done formally before. The results are hardly surprising, though, because we know that the regression estimates $\hat{\bm{\alpha}}$ are BUE, thus, they are unbiased and have minimum variance among all unbiased estimators. Within this class (of unbiased estimators), t -- tests based on the OLS estimators will thus be optimal. The proofs presented in this part, however, do not require that the test statistics be unbiased; rather, proving that a test is UMP relies on the notion of tests having Neyman -- structure, that is having nominal type I error on the boundary of the parameter space for each observed value of the nuisance statistics. The meaning of ``unbiased'' in UMPU is different, meaning that the power function is below the significance level $\alpha$ for parameter values in $H_0$, and above $\alpha$ on $H_1$. For more details on these concepts, we refer the reader to \cite{bhattacharya, lehmann}.

In the next part, we will investigate how equivalent formulation of multiple regression change the   interpretation of effect sizes (Section \ref{multicoll}) and can lead to improved power for testing parameters (Section \ref{gains}).

%\vspace{0.5cm}

\section{Testing beyond multiple regression} \label{multicoll} 
%\vspace{0.5cm}
Power loss in linear models is closely linked to the presence of multicollinearity, by which we mean the covariance structure among predictors. This  is not modeled as part of the regression equation, which conditions on the predictors in $M$. Two options have emerged when seeking to circumvent the multicollinearity problem in the original model: (i) keep the original features, but penalize the estimates; or (ii) find a meaningful set of new features, typically orthogonal. %This resolves the problem of multicollinearity, but leaves the problem of interpretation. 
We look at these two approaches, in turn.

\subsection{Ridge regression} 
Starting from the multiple regression model $Y = M \bm{\alpha} + \epsilon$, \cite{hoerl} adapted the OLS estimator $\hat{\bm{\alpha}} = (M^\intercal M)^{-1} M^\intercal Y$ by adding a ``ridge'' to the diagonal of $M^\intercal M$, making the estimator
\begin{equation}  \label{ridge}
\hat{\bm{\alpha}}_{ridge} = (M^\intercal M + k I_p)^{-1} M^\intercal Y.
\end{equation}
This improves the stability of estimates and alleviates the impact of multicollinearity, at the expense of $\hat{\bm{\alpha}}_{ridge}$ being biased. The higher the ridge parameter $k$ is, the more the coefficient estimates approach $M^\intercal Y/k$, namely, the lengths of projections of the data vector $Y$ in the directions of each regressor, but scaled downward by factor $k$. In the literature on ridge regression, multicollinearity is seen as the ill-conditioning of matrix $M^\intercal M$, which is measurable via its eigenvalues (see \cite{hoerl, halawa}, etc). Specifically, if the eigenvalues of $M^\intercal M$ are, in order, $\lambda_{max} > \lambda_{(p-1)} > ... >  \lambda_{(2)} >\lambda_{min}$, eigenvalues close to zero imply a high degree of linear dependence between the columns of $M$. A statistic proposed by \cite{liu} to measure multicollinearity is the condition number $CN = \sqrt{\lambda_{max}/\lambda_{min}}$.
A condition number between 30 -- 100 is indicative of moderate/strong multicollinearity, while values greater than 100 correspond to severe multicollinearity.

By alleviating the symptoms of multicollinearity between predictor variables, ridge estimates have the potential to improve power over the t--tests in OLS. A similar t--test can be constructed for the parameter of interest, %which has estimate given by \ref{ridge}, with variance
%\[
%Var(\hat{\bm{\alpha}}_{ridge}) = \sigma^2 (M^\intercal M + k I_p)^{-1} M^\intercal M (M^\intercal M + k I_p)^{-1}.
%\]
%The resulting t--statistic is 
$t_i = \hat{\alpha}_{i, ridge}/s.e.(\hat{\alpha}_{i, ridge})$, where the standard error is the square root of the $i$--th diagonal element of  $Var(\hat{\bm{\alpha}}_{ridge})$. Similar to multiple regression, $t_i$ is distributed as Student--t with $n-p$ degrees of freedom, assuming the initial variables have been centered and model (\ref{ridge}) contains no intercept. 
%\vspace{-3cm}
\subsection{Interpretation of Gram--Schmidt and multiple regressions}\label{interpretation}
In contrast to ridge regression, GS regression completely eliminates correlation among predictors by orthogonalizing the predictor set. %In this subsection, we look at what the parameter estimates mean, and when this would be an appropriate model for data. 
Moreover, compared to other orthonormal transformations, GS features have a special claim to interpretability. To see this, we need to better understand the structure among independent variables. 
Structural equation models (SEMs) attempt to fully describe this structure  via systems of equations involving the predictors. %, including (possibly) distributional assumptions on error terms. 
The GS decomposition naturally corresponds to a certain set of equations that define the ``covariance structure'' among the original predictors in terms of the remainders $\bm{x}$. %We perform the GS orthogonalization for a particular ordering of the original variables given by permutation $\pi : (1, ..., p) \rightarrow (1, ..., p)$, such that the first variable in the GS sequence is $1$ from the original list, the second is $2$, and so on. 
From Algorithm \ref{algo}, we can write the following system for observed variables:
\begin{equation} \label{SEM}
\begin{cases}
\bm{m}_{1} &= \alpha_{11} \bm{x}_1 \\
\bm{m}_{2} &= \alpha_{21} \bm{x}_1 + \alpha_{22} \bm{x}_2\\
...\\
\bm{m}_{p} &= \alpha_{p1} \bm{x}_1+ \alpha_{p2} \bm{x}_2 +... + \alpha_{pp} \bm{x}_p \\
\bm{Y} &= \alpha_{1} \bm{m}_{1}+  \alpha_{2} \bm{m}_{2} +... + \alpha_{p} \bm{m}_{p} + \bm{\epsilon}\\
 &= (\alpha_{1} \alpha_{11} + \alpha_{2}\alpha_{21} + ...+ \alpha_{p} \alpha_{p1}) \bm{x}_1 + ...+ \alpha_{p} \alpha_{pp} \bm{x}_p  + \bm{\epsilon}.
\end{cases}
\end{equation}
Here, the SEM and related literature often interprets $\bm{x}_1, ..., \bm{x}_p$ as random latent factors, possibly coming from an independent standard normal distribution (see, e.g., \cite{goldberger}), if the original data have been appropriately centered. As observed by \cite{cross}, this particular SEM structure can further be identified with a directed acyclic graph (DAG), where each variable influences both the next variable as well as the outcome $Y$, both directly and indirectly, via all variables downstream of it (Figure \ref{fig-dag}).  For the rest of the paper, we assume a known temporal causal order of predictors, which we use in the orthogonalization.
This may be given by expert knowledge, or by applying methods of causal discovery on the original variables. %Either way, the correct causal specification of the model would ensure a meaningful interpretation of effect sizes.  
\begin{figure} [htb]
\centering
\includegraphics[width=1\linewidth]{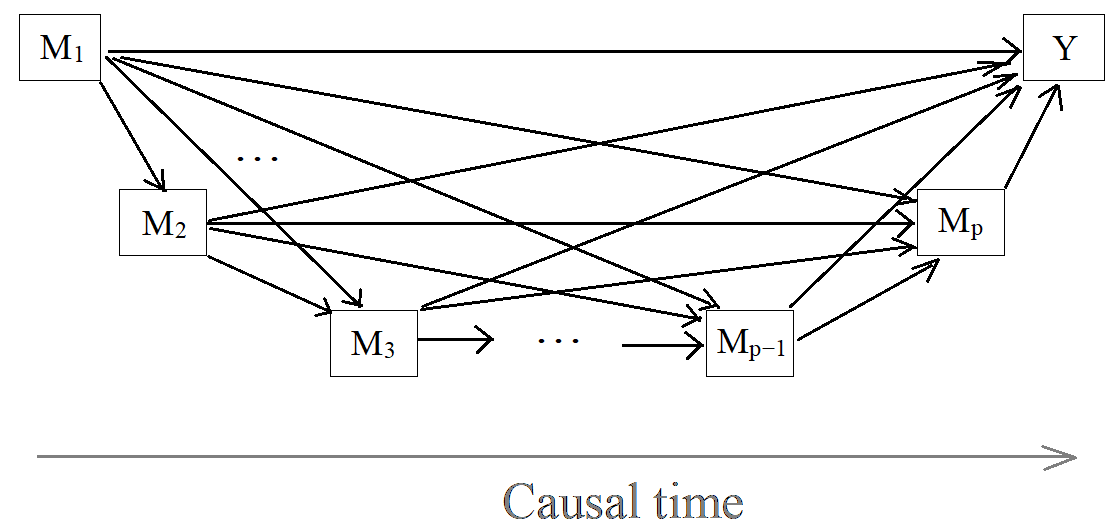}
\includegraphics[width=1\linewidth]{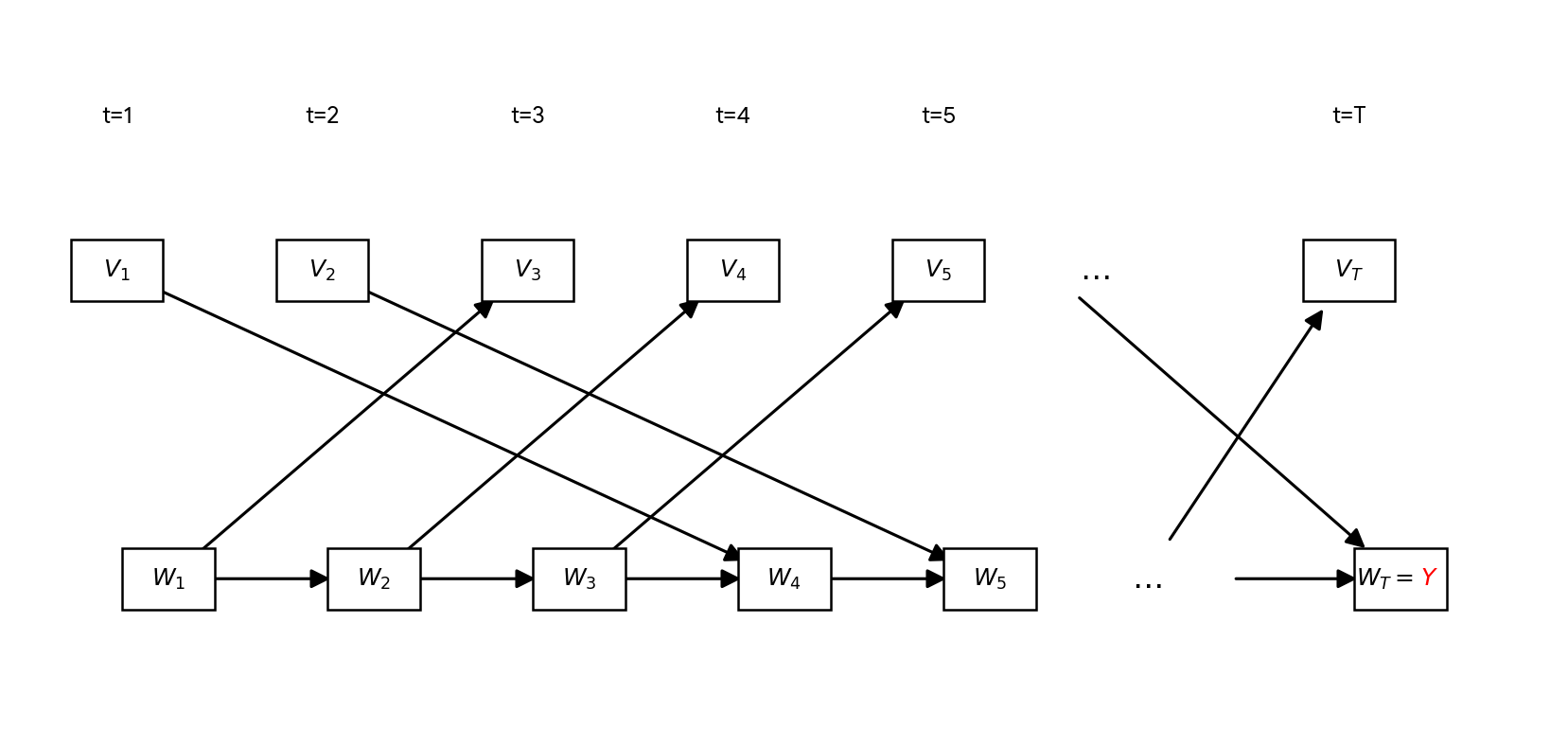}
\caption{ \textit{(Top)} Directed acyclic graph showing the one--way flow of information in the direction of causal time. \textit{(Bottom)} Feedback loop between two variables, where the causal effect takes a constant lag to manifest (graph diagram generated using Gemini \cite{google} based on a conceptual layout request). }
\label{fig-dag}
\end{figure}

It is important to remember that this interpretation in model (\ref{SEM}) with the additional assumption of $\bm{x}_i$ being random is neither implied by the regression equation, nor a condition for using GS regression. Rather, it is an assumption we make to be able to draw causal conclusions. %Indeed, in the regression space predictors can simply be thought of as fixed design vectors. However, 
The extra SEM structure is illuminating in helping us better understand which effect, specifically, can be ascribed to a treatment, after adjusting for confounders. We may distinguish between direct and indirect effects of a variable on the outcome. For instance, the total effect of factor $\bm{x}_1$ on $y$ in the final equation may be thought of as $\partial y/\partial x_1$, as per the causal inference literature (see, e.g., \cite{pearl}). This can be decomposed into its direct effect ($\alpha_{1} \alpha_{11}$) and indirect effect ($\alpha_{2}\alpha_{21} + ...+ \alpha_{p} \alpha_{p1}$). By contrast, latent factor $\bm{x}_p$ only has a direct effect  ($\alpha_{p} \alpha_{pp}$). To measure the effect of the original predictors $\bm{m}_i$ on the outcome, we still use  partial derivatives $\partial y/\partial m_i$. %, which measure the change in outcome %function $y=f(m_1, m_2,..., m_p, \epsilon)$ caused by a unit change in   $m_i$, assuming no change in the other variables.
However, as the $m$ predictors are linked, it is impossible to change just $m_i$ without affecting the others. Still, it is both possible and realistic to require that  a change in $m_i$ will not affect any past variables $m_j, j<i$. Such direct intervention 
%to see whether and how this change is possible, we need to look at the causal architecture in more detail. For the  $i$--th predictor $m_i$ in model  (\ref{SEM}), 
%$$ \bm{m}_{i} = \alpha_{i1} \bm{x}_1+ \alpha_{i2} \bm{x}_2 +... + \alpha_{i,i} \bm{x}_i. $$
%The understanding here is that we can intervene directly 
to change $m_i$ by one unit can be accomplished  via a $1/\alpha_{i,i}$ change in $x_i$, without changing any of the other exogenous factors $x$. The effect on $Y$ would be $(\alpha_{{i}} \alpha_{ii} + \alpha_{{i+1}}\alpha_{i+1,i} + ...+ \alpha_{p} \alpha_{pi}) /\alpha_{i,i}=  \beta_i/\alpha_{i,i}$. % Factors $m_{1}, m_{2},...,m_{{i-1}}$ are upstream from $m_i$ and can be held constant.  All variables downstream of $m_i$, namely $m_{{i+1}}, m_{{i+2}},...,m_{p}$ will have to change due to the change in $x_i$. Thus, that the magnitude of the effect depends heavily on the temporal order of our variable of interest, as well as on the correlation structure it has with its downstream variables. 

By contrast, in multiple regression the effect size $\partial y/\partial m_1=\alpha_1$, because there is no assumed structure among the independent variables. This means that all features have only a direct effect on the outcome, which is not consistent with  model (\ref{SEM})! If we were to postulate an SEM for predictors that is consistent with the multiple regression interpretation, it should be similar to the following:
\begin{equation} \label{SEM_naive}
\begin{cases}
M_i &= \alpha_{i1} X_1 + \alpha_{i2} X_2 +...+\alpha_{ip} X_p + \sigma_M \epsilon_i, \quad \text{for all } i=1..p \\
Y &= \alpha_1 M_1+  \alpha_2 M_2 +... + \alpha_p M_p + \epsilon\\
 &= (\sum_i \alpha_i \alpha_{i1}) X_1 + (\sum_i \alpha_i \alpha_{i2}) X_2 + ...+ \sigma_M \sum_i \alpha_i\epsilon_i + \epsilon,
\end{cases}
\end{equation}
where $X_1, ..., X_p$ and $\epsilon_1, ..., \epsilon_p$ are independent with mean zero and variance one.
In this model, each predictor $M_i$ has an idiosyncratic component $\epsilon_i$ (which can be intervened upon), and the observed correlation structure is driven by the $X$ latent factors (which cannot). A unit change in $M_i$ that does not affect the other latent factors can only come about  by a $1/\sigma_M$ change in $\epsilon_i$, which has a direct effect of $\alpha_i$ on $Y$. 
Thus, there are implicit causal assumptions underlying the multiple regression model, namely that each predictor can be modified witout any spillover effect into the other predictors, and that all $\epsilon_i$ variables are causally simultaneous. This would be reasonable if $\epsilon_i$ were error terms, but then such terms cannot usually be acted upon. Overall, one would be hard--pressed to find   real--world situations where this causal logic was substantively motivated.

The statistical properties of the effect size estimate in our temporal model are summarized in the following

 \begin{proposition}\label{causal}
Assuming the SEM (\ref{SEM}) corresponding to the GS decomposition in Algorithm \ref{algo}, the various causal effects associated with  the $i$-th original predictor are given by:
\begin{enumerate}
\item The total effect of $M_i$ is  $\frac{\beta_i}{Q_{ii}}$, estimated as $  \hat{\beta}_i/Q_{ii} \sim N(\frac{\beta_i}{Q_{ii}}, \frac{\sigma^2}{Q_{ii}^2}).$
\item The direct effect $=\alpha_i=\bm{q}_i^\intercal \bm{\beta}$, with distribution $\bm{q}_i^\intercal \hat{\bm{\beta}} \sim N(\bm{q}_i^\intercal \bm{\beta}, \sigma^2 ||\bm{q}_i||^2)$;
\item The indirect effect $=\frac{\beta_i}{Q_{ii}}-\bm{q}_i^\intercal \bm{\beta}$, with distribution $ \hat{\beta}_i/Q_{ii}-\bm{q}_i^\intercal \hat{\bm{\beta}} \sim N(\frac{\beta_i}{Q_{ii}} - \bm{q}_i^\intercal \bm{\beta}, \sigma^2 ( ||\bm{q}_i||^2 - 1/Q_{ii}^2 ) )$
\item The proportion of the effect that is mediated $=\frac{\text{indirect effect}}{\text{total effect}}=1-\frac{\bm{q}_i^\intercal \bm{\beta} Q_{ii}}{\beta_i}$. The estimator $W=1- \frac{\bm{q}_i^\intercal \hat{\bm{\beta}}Q_{ii}}{\hat{\beta}_i} $ is unbiased, with sampling distribution  $P(W \leq w) = 1-F(-\frac{w}{\sqrt{Q_{ii}^2 ||\bm{q}_i||^2-1}} )$, where function $F$ is given in the Appendix.
\end{enumerate}
\end{proposition}
For case 4 in particular, it's important to note that the proportion mediated can be negative, or greater than one, in the case of inconsistent mediation (direct and indirect effects have opposite signs). Confidence intervals for the proportion mediated $(w_l, w_u)$ can be computed numerically as quantiles of the sampling distribution, i.e.,  by solving $P(W <w_l)=\alpha/2$ and   $P(W <w_u)=1- \alpha/2$.

\subsection{Tests for causal effects}\label{causal.tests}

In current practice in mediation analysis, the  effect through an edge of the DAG is typically estimated as a parameter, and the various effects for a feature %(especially its mediated effect via all downstream variables) 
are computed by summing the products of parameters along all paths to the outcome. 
For example, in  the single mediator model (which roughly corresponds to model (\ref{SEM}) with $p=2$), we have $M_1$ affects $Y$ through mediator $M_2$. The indirect (or mediated) effect is the product of slopes in the mediation pathway ($M_1 \rightarrow M_2$ and  $M_2 \rightarrow Y$), namely $\alpha_{21} \alpha_2$. To test a non-null indirect effect, one may use the Sobel test (\cite{sobel}), also referred to as the delta method, which obtains a z-test by dividing  $\hat{\alpha}_{21} \hat{\alpha}_2$ by its standard deviation. The latter is computed assuming that the product of estimates is normal, which leads to the test being conservative (Mackinnon et al 1995). Another way is the joint significance test, which concludes significance in mediation if each of   $\hat{\alpha}_{21}$ and  $\hat{\alpha}_2$ are significant. This test works well, but its downside is that this method does not produce a confidence interval (following the review from \cite{kenny}). Perhaps the best method for testing and CI estimation is the bootstrap, however, this is computationally heavy. 

Beyond the three variable mediation model,  \cite{orourke} have investigated the parallel two mediator case, which is similar (though not the same) as  model (\ref{SEM}) with $p=3$. Following the same reasoning of using slope products described above has led them to fairly complex formulas, which would likely be difficult to extend for higher $p$. 
This is due to path counting, which can become combinatorially prohibitive.  In this case,  calculations (including for sample size) are often only possible via bootstrapping or simulation, both of which are intractable.

With this new vector geometric approach to solve causal model (\ref{SEM}), we can more easily test hypoteses related to causal effects, such as whether there is a significant total or indirect effect for variable $i$, or whether  a significant proportion of the total effect of $M_i$ is mediated.
Using Proposition \ref{causal}, we can form test statistics and null distributions for  hypotheses  $H_0^\text{Total}:\beta_i=0$, and $H_0^\text{Indirect}$ \& $H_0^\text{Prop}: \alpha_i=\beta_i/Q_{ii}$. To test against the corresponding right--sided alternatives, we employ the following test statistics:

\begin{align}
Z^\text{Tot} &= \frac{ \hat{\beta}_i}{ s},\\
Z^\text{Ind} &= \frac{ \hat{\beta}_i/Q_{ii}-\bm{q}_i^\intercal \hat{\bm{\beta}} }{ s \sqrt{||\bm{q}_i||^2 - \bm{q}_i }}, \ \text{and}\\
W &=1- \frac{\hat{\alpha_i}Q_{ii}}{\hat{\beta}_i}.
\end{align}

The usual z--test is used with the first two statistics, which are standard normal. The p--value of the  fraction mediated being greater than observed is $\Phi \left(-\frac{\hat{\beta}_i A_0}{s}  , - \frac{\hat{\beta}_i}{s}, A_0\right) + \Phi \left(\frac{\hat{\beta}_i A_0}{s }, \frac{\hat{\beta}_i}{s}, A_0 \right)$, 
where $A_0=- \frac{W_0}{\sqrt{Q_{ii}^2 ||\bm{q}_i||^2-1 + W_0^2}}$, $W_0$ is the observed fraction $W$, and $\Phi$ is the bivariate normal CDF. %The null distribution for $W$ is given explicitly in the Appendix.
Note that division by $Q_{ii}$ %$||\hat{\bm{r}}_i||$ 
is only necessary due to our choice to normalize the new basis in Algorithm  \ref{algo}. This step could be skipped had we not chosen to normalize, in which case the residuals would form the new basis ($\bm{x}_i = \bm{r}_i $) and $\alpha_{ii}$ would be set to 1. The effect size would then simply be $\beta_i$. In terms of hypothesis testing, the tests are equivalent whether  we normalize or not. %: a test for the effect size $\beta_i/||\bm{r}_i||$ being different from zero for is equivalent to the same test on $\beta_i$.

Compared to the delta or bootstrap methods, computing the indirect effect as in Proposition \ref{causal} benefits from having an exact, closed form density, which is ideal for CI and power computations. However, the true potential  of the GS approach to effect decomposition is that it replaces path counting in mediation analysis. %
Unlike the traditional mediation approach of summing slope products over all paths, in GS the total effects are computed from the outcome regression alone,  for any value of $p$ (\cite{cross}). 
%\subsection{Orthogonal vectors as an interpretation of causal inference}
%The vectors $\bm{m}_i^\perp$  map a geometric solution to the causal inference problem expressed algebraically by SEM system (\ref{SEM}), for any value of $p$. Explicitly, the total effect is the component of the data along the direction of the innovation given by the $i$--th variable with respect to its temporal predecessors,  $\bm{m}_i^\perp$. From this mapping, the direct and indirect  effects of each variable can be identified in terms of the coordinates of data vector $Y$ in this ``solution'' orthogonal basis. Treating the predictor variables as fixed, the coordinates are all independent as components of a spherical data vector, making the effect distributions easy to work with. %Thus, assuming a causal time chronology $\bm{x}_1,    \bm{x}_2,..., \bm{x}_p$, ...
This alternative, geometric solution to the causal inference problem expressed algebraically by SEM system (\ref{SEM}) is made possible by figuring out the directions $\bm{m}_i^\perp$ for each innovation term.  
%for any value of $p$. Explicitly, the total effect is the component of the data along the direction of the innovation given by the $i$--th variable with respect to its temporal predecessors,  $\bm{m}_i^\perp$. From this mapping, the direct and indirect  effects of each variable can be identified in terms of the coordinates of data vector $Y$ in this ``solution'' orthogonal basis. 
Treating the predictor variables as fixed, the total effects are all independent as (scaled) coordinates along these directions of the spherical data vector $Y$ (under the null model). The addition or removal of predictors will have no effect on these tests, except through the standard deviation estimate $s$.

%Call this geometric causal inference?
\subsection{Power gains of  testing total effects over multiple regression} \label{gains}
Besides its relevance to mediation methods, the temporal causal optic also confers a potentially important power advantage when testing total effects instead of the direct effects implicit in classical regression. This is due to no longer losing power to multicollinearity.  The following theorem explores the conditions under which   a power difference exists when testing the coefficient of a feature  under the GS and multiple regressions. It then computes an equivalent sample size under the two models to attain the same power. The intent of this calculation is to demonstrate the utility of the GS method in study design, for example in the context of a follow--up study (of size $k n_0)$ to  a pilot  (of size $n_0$). Here, we consider a larger study to be simply a scaled--up version of the pilot, including $k$ replicates for each row of the original design matrix. Let us first define the following quantity:
\begin{definition}
 Define $\Delta_i =\frac{\beta_i ||\bm{q}_i||}{  \bm{q}_i^\intercal \bm{\beta}}$ as a diagnostic tool for testing variable $\bm{x}_i$  in Algorithm \ref{algo}. This can be equivalently written $\Delta =\frac{\beta_i ||\bm{q}_i||}{  \alpha_i}$, or $\Delta =\frac{||\bm{q}_i|| Q_{i*} \bm{\alpha} }{  \alpha_i}$. 
\end{definition}
%A couple of explanatory remarks to the previous result are in order.
We now formulate the main result concerning power between MR and GS.

 \begin{theorem}\label{thm5}
In the following two parameterizations of the same regression model: 
\begin{align*} \bm{Y} &=\alpha_1 \mathbf{m}_{1}+\alpha_2 \mathbf{m}_{2}+...+\alpha_p \mathbf{m}_{p}+ \bm{\epsilon}, \text{and}  \tag{A}\\
\bm{Y} &=\beta_1 \mathbf{x}_{1}+\beta_2 \mathbf{x}_{2}+...+\beta_p \mathbf{x}_{p}+ \bm{\epsilon} \tag{B}
\end{align*}
where $\bm{\epsilon} \sim N(0,\sigma^2 I)$ and  $M=XQ$ is the Gram--Schmidt decomposition of design matrix $M$ (with $X$ orthonormal, and $Q$ upper triangular); let tests $\phi_A, \phi_B$ for $H_{A0}: \alpha_i \leq 0$ vs $H_{A1}: \alpha_i > 0$; and $H_{B0}: \beta_i \leq 0$ vs $H_{B1}: \beta_i > 0$, respectively, be defined via the usual t statistics $V_A=\frac{ \hat{\alpha}_i}{ \text{s.e.}( \hat{\alpha}_i)}  \quad \text{and} \quad V_B=\frac{ \hat{\beta}_i}{ \text{s.e.}( \hat{\beta}_i)}$. 
Then:
\begin{enumerate}[(a)]
\item the power of  $\phi_B$ is higher  than the power of $\phi_A$ iff $\beta_i >   \bm{q}_i^\intercal \bm{\beta}/||\bm{q}_i||$. 
\item In two planned studies of sample sizes $n_A$ and $n_B$ to be analyzed via models A and B, respectively,  a one-sided test of variable $i$ in each model is asymptotically equivalent in terms of power iff $\frac{n_A}{n_B}=\Delta_i^2$ and $\alpha_i, \beta_i$ have the same sign.
\end{enumerate}

\end{theorem}

This theorem suggests that $\Delta_i$ is an important quantity related to multicollinearity. If, in addition, we knew that $\alpha_i >0$, then part (a) says that  the GS regression will lead to a more powerful test for the first predictor compared to (naive) MR if and only if $\Delta_i > 1$. %This  that summarizes how predictor geometry and effect sizes combine to make this approach advantageous when planning studies
To understand this result more vividly and see how it connects to previous measures of multicollinearity, a couple of remarks are in order. Assume, for now, that  the original regressors $\bm{m}_i$ have been centered. %standardized. This means that $M^\intercal M$ is their sample correl matrix.

 \begin{remark}\label{rem1}
The quantity  $||\bm{q}_i||$ can be shown to be equal to $\sqrt{VIF_i}/||\bm{m}_i||$. This allows us to write $\Delta_i  =\frac{\beta_i \sqrt{VIF_i}}{  \alpha_i ||\bm{m}_i||}$.
\end{remark}

This shows that $\Delta$ is a more meaningful quantity than VIF, because while VIF only relates to the predictor set, $\Delta$ incorporates information about the outcome as well.

 \begin{remark}\label{rem2}
If all parameters in (\ref{SEM}) are positive, model B will always be more powerful than model A for testing the coefficient of the first predictor $\bm{m}_1$.
\end{remark}
The results in this section imply that recasting a MR in temporal causal terms and testing for total effects (where this makes sense), may be desirable as a general practice due to the power advantage, even if one is not particularly interested in causal inference.

\section{Further decomposition of indirect effects  and application to cross--lagged feedback loops}\label{feedback}

 %Relying on slope products along paths leads to fairly complex formulas for DAGs with depth    $p=3$  (similar to the parallel two mediator case in \cite{orourke}), and become combinatorially prohibitive for higher $p$. The GS approach has no problem computing an indirect effect for any value of $p$, once it identifies the unique effect (or contribution) of variable $i$ as  $\bm{m}_i^\perp$. 

The notion of indirect effect can be decomposed more atomically into  how much the effect of one variable is mediated by a single other variable. This can be seen by rewriting system (\ref{SEM}) to show the effect flow between variables $M_i$. 

\begin{equation} \label{SEM2}
\begin{cases}
\bm{m}_{1} &= \alpha_{11} \bm{x}_1 \\
\bm{m}_{2} &= \gamma_{21} \bm{m}_1 + \alpha_{22} \bm{x}_2\\
...\\
\bm{m}_{p} &= \gamma_{p1} \bm{m}_1+ \gamma_{p2} \bm{m}_2 +... + \gamma_{p2} \bm{m}_{p-1} + \alpha_{pp} \bm{x}_p \\
\bm{Y} &= \alpha_{1} \bm{m}_{1}+  \alpha_{2} \bm{m}_{2} +... + \alpha_{p} \bm{m}_{p} + \bm{\epsilon}\\
 &= \beta_{1}  \bm{x}_1 + \beta_{2}  \bm{x}_2 + ...+ \beta_{p} \bm{x}_p  + \bm{\epsilon}.
\end{cases}
\end{equation}

Although systems  (\ref{SEM}) and  (\ref{SEM2}) are equivalent, the latter is often a better mechanistic description of physical processes, where variables depend on each other through relationships of the form $E\left[M_i\right] = f_i(M_1, M_2,...,M_{i-1})$, as opposed to only via their innovation terms. This is intuitive in economic systems: for example, an investment in scholarship  funding will allow students to dedicate more time to academics, but also spend money on extracurricular activities. Both of these are likely to impact future outcomes (such as earning potential or quality of life), but there would be an interest to know which arm is most beneficial for student success, so that  future investments can be more specifically targeted.
 In this setting, the innovation in a variable (i.e., the extra investment) is transmitted to all downstream variables both directly, but also indirectly, through its mediators. 
Here, we will further develop the algebra of direct/indirect effects by showing how the indirect effect propagates via its mediators.

 \begin{proposition}\label{indirect}
Assuming the SEM (\ref{SEM}) corresponding to the GS decomposition in Algorithm \ref{algo}, the causal effects associated with  the $i$-th original predictor are given by:
\begin{enumerate}
\item The indirect effect of $M_i$ on $Y$ mediated by  $M_j, j>i$ is  $\frac{\beta_j Q_{ij}}{Q_{ii}Q_{jj}}$ (for a one unit intervention on $M_i$ via its innovation term). This makes the fraction of total effect mediated  $\frac{\beta_j Q_{ij}}{\beta_i Q_{jj}}$.
\item The indirect effect of predictor set $\{M_i : i \in S=\{s_1, ...,s_k\}\}$ mediated respectively by predictor set $\{M_j : j \in U=\{u_1,...,u_k\}\}$ via the correspondence $SU =\{(s_1,u_1),...,(s_k,u_k)\}$, where $s_i<u_i$, is given by  $\sum_{(i,j)\in SU} \frac{\beta_{j} Q_{ij}}{Q_{ii}Q_{jj}}$ (assuming a unit intervention on each variable in $S$). 
\item The corresponding vector contribution of  $M_i$ to $Y$, mediated by  $M_j$ is  $\frac{\beta_j Q_{ij}}{Q_{jj}}\bm{x}_i$.  The vector contribution of  set $S$ to $Y$, mediated respectively by  $U$ is  $\sum_{(i,j)\in SU} \frac{\beta_{j} Q_{ij}}{Q_{jj}} \bm{x}_i$. %(for a unit change in all $M$ variables in $S$). 
\item Define the average fraction of total effect of variables in $S$ mediated by variables in $U$ by the ratio of scalar effects $D^1=\frac{\sum_{(i,j)\in SU} \frac{\beta_{j} Q_{i j}}{Q_{ii}Q_{jj}}}  {\sum_{(i,j)\in SU} \beta_{i} / Q_{ii}} $, and the ratio of squared norms of vector contributions vs. vector total effects, $D^2=\frac{\sum_{(i,j)\in SU}  \beta_{j}^2 (Q_{i j}/Q_{jj})^2 }  {\sum_{(i,j)\in SU} \beta_{i}^2 } $.
%$D^2=\frac{\sum_{(i,j)\in SU} \beta_{j}^2 (\frac{ Q_{i j}}{Q_{ii}Q_{jj}})^2 }  {\sum_{(i,j)\in SU} \beta_{i}^2 / Q_{ii}^2} $. %This is approximately distributed as a scaled \(F\)-distribution $\frac{c_1 \nu_1}{c_2 \nu_2}F(\nu_1, \nu_2)$, where the constants are given in the Appendix.
\end{enumerate}
\end{proposition}
%As in Proposition \ref{causal}, the computation of mediated indirect effects avoids path counting and have an exact distribution.
Note that the vector total effects assumed in item 4 are $ \beta_{i}  \bm{x}_i$. One advantage of these vector definitions for total and indirect effects is that they are unitless, as opposed to the traditional definition of $\partial Y/ \partial M_i$, which depends on the unit of measurement. The latter makes comparison among studies more difficult, as changing units (e.g., miles to kilometers, or one currency to another) changes the reported effect size. The vector definition makes even more sense when one needs to compare between effects of different variables, such as in the fractions $D^1$ and $D^2$, which report average mediation across a set of variables.

\begin{remark}
The mediated indirect effect can be expressed as $\propto \partial Y/ \partial M_j^\perp \times  \partial M_j/ \partial M_i^\perp$. The partial derivatives can further be identified  as $  (\bm{Y} \cdot  \bm{m}_i^\perp) \times  (\bm{m}_j \cdot  \bm{m}_i^\perp)$.
\end{remark}

In general, indirect effects through different channels do not add up to the total indirect effect, because overlapping paths would be counted multiple times. However, in the special case when paths do not cross we have the following
\begin{consequence}
If the indirect effect of $M_i$ is fully mediated through predictor set $\Omega =\{M_j: j \in \{j_1, j_2,...,j_K\}, j>i \}$, such that paths from  $M_i$ to $Y$ through any member of $\Omega$ do not intersect with paths through any other  member of $\Omega$ (except for beginning and end points), % and none of the paths share any common points (except for beginning and end), 
then the total effect of $M_i$ can be decomposed as:
\begin{equation}  \label{indirect.decomp}
\text{Total}_i = \text{Direct}_i + \sum_{k=1}^K \text{Indirect}_{i \rightarrow j_k \rightarrow Y}.  
\end{equation}
\end{consequence}
%Statistically, the total effect estimate from \ref{indirect.decomp} using the mediated indirect effect estimates from Proposition \ref{indirect} is independent of the total effect estimate  $\hat{\beta}_i / Q_{ii}$ from Proposition \ref{causal}, given $\hat{\alpha}_i$. This means that, in this special case, a more efficient estimate of the total effect of  $M_i$ is available. 
The decomposition of indirect effects is a novel contribution, as is the use of vector components, which retain more information in the case of inconsistent mediation compared to scalar effects. 
A practical use of these quantities is in the area of causal discovery, where detecting a large mediated indirect effect suggests the existance of causal links between the originating variable (or set) and the mediator variable (or set). We illustrate this in the context of a feedback loop next.%Before showing an application, we derive the corresponding hypothesis tests.  

Suppose that the variables in a temporal DAG correspond to longitudinal observations on the same $k$ state variables $W_{1,t}, W_{2,t},...,W_{k,t}$, at different points in time. Suppose, for simplicity, that each variable is observed at each time point, and let $d_{ij}, i,j=1,...,k$ be causal lags such that variable  $W_{i,t}$ has a causal effect on $W_{j,t+d_{ij}}$, for every time $t=1,...,T$. %By convention, if there is no causal effect between two time series, set $d_{ij}=\infty$. 
 For a complete identification of this graph with a temporal DAG, we assume a certain causal ordering at each time step $t$, with $W_1$ being the ``oldest'' and  $W_k$ the ``newest'' variable at $t$. What we have constructed is a feedback loop, and one goal of causal discovery in this setup is to determine the integer lags $d_{ij}$ from data. An example of such a DAG with two time series %$X_t$ and $Y_t$ 
is shown in Figure \ref{fig-dag}. 
In the case of linear causal relationships, the solution we propose is to set up the problem in terms of temporal DAGs and use the results in Proposition \ref{indirect} to estimate the mediated effect from one time series to another, for each lag. We expect that the correct lag will maximize the fraction of one time series' indirect effect mediated by the other. The steps to take in this analysis are as follows:
\begin{enumerate}
\item First, decide on an outcome; this is typically the last causal variable, $W_{k,T}$. Notice that all other predictor variables plus the outcome constitute a temporal DAG with some of the arrows missing (i.e., with  zero effect).
\item Define the sets $W_{1,+d}=\{W_{1,d}, W_{1,d+1},..., W_{1,T}\}$, and similarly $W_{2,+d}, W_{3,+d},...$. Notice that each variable $W_{i,t}$, where $ t<T-d_{ik}$ may only have an indirect effect on the outcome, mediated through other variables in one or more of the sets $W_{1,+0},...,W_{k,+0}$. 
\item Compute the average fractions of effects of the $W_{i,+0}$ series  mediated by $W_{j,+d}$ series via measures $D^1$ or $D^2$ in Proposition \ref{indirect} (trimming excess variables as needed to achieve paired correspondence). Call these quantities $D^1_{i,j,d}$ and $D^2_{i,j,d}$, for all lags $d = 0,...,d_{\text{max}}$.
\item For each pair of time series, identify the significant lags by comparing  $D^{1,2}_{i,j,d}$ against the 95th quantile of its null distribution. This can be obtained by simulation methods (see Appendix for details). 
\end{enumerate}

%This analysis is similar to lag determination in time series, which uses the autocorrelation function (ACF) in a similar way. %to detect significant lags, or using economics...
The conceptual novelty of our approach is first in detecting causal relationships using indirect effects  as opposed to correlation. The latter is standard in causal discovery,  for instance in the Peter Clark algorithm, which tests for marginal and conditional independence between pairs of variables via partial correlation   (see \cite{spirtes} and more recent adaptations in, e.g., \cite{peters}).  In the context of lag determination specifically, the analysis presented above is similar to the use of the autocorrelation function (ACF) in time series. However, the ACF relies on the stationarity assumption, as does most of the literature on causality in time series, which is based on the vector autoregressive (VAR) framework  of \cite{sims1980}. Our approach to effect decomposition has the advantage that it does not require stationarity, meaning that it can analyze processes wherein the mechanics of causation may change from one period to the next. This often happens in systems that are subject to regime changes. 

Yet another difference with autoregressive models it the type of data they are meant to analyze. VAR--type models are typically meant to be used with single--observation time series, e.g., the data $W_{1,t}, W_{2,t},...,W_{k,t}$ may represent daily observations of $k$ different macroeconomic variables, observed over a typically long time frame. Whereas the GS approach can handle repeated measures over  both individuals and time, making it particularly relevant to analyzing longitudinal population studies, in which $k$ outcomes  are typically observed in many individuals over a limited number of time points.
%Another advantage of this method is being able to handle a two--dimensional data structure over both individuals and time. 
This vector structure of the data across individuals has allowed defining causal effects in a high--dimensional vector space, which can be advantageous for inference (as already mentioned) compared to the scalar definitions of effect sizes. The latter are often referred to as impulse responses in the VAR literature (see, e.g., \cite{jorda, dufour}), and are probably the only way of looking at the issue in econometrics.

%Results: show fraction of each lag detected as table, and put a representative figure similar to ACF in Appendix.

\section{Applications}\label{sim}

\subsection{Simulations of power gains over regression} \label{study-pow}
We generate variables $M_1, M_2,..., M_p$ with a known correlation structure, then simulate a continuous outcome $Y$ conditional on design matrix $M$. In the notation of model (\ref{SEM}), the data generating equations are:

\[
\begin{cases}
M_1=Z_1 \\
M_i=\rho Z_1 +  Z_i, \quad  \text{for $i=2..p$, and}\\
Y = \frac{1}{p} M_1 + ... + \frac{1}{p} M_p +\sigma \epsilon,
\end{cases}
\]
where $Z_1, Z_2,..., Z_p,$ and $\epsilon$ are i.i.d. standard normal variates.
Thus, $\rho$ controls the correlation between the independent variables, and $\sigma$ controls (indirectly) the correlation between all predictors and the outcome. A small $\sigma$ will increase the effect size for all variables; meanwhile, a large $\rho$ increases the multicollinearity. We consider a combination of scenarios with $\rho$ taking values in $\{-0.25, 0.25, 0.5 \}$; $\sigma$ covers the positive range from 1 to $\infty$, and the number of predictors $p$ is in $\{3, 5, 15\}$.  For  $N=1000$ replicated studies of size $n=200$ samples, we obtain empirical power at the $5\%$ level for testing that the coefficient of $M_1$ is positive versus zero. 

The models used are: (a) naive multiple regression of the centered outcome on the scaled and centered $M_i$ variables, without an intercept; (b) GS regression which orthogonalizes the centered and scaled input matrix $M$ around the first variable ($M_1$); and (c) ridge regression, using the same input as (a). The tuning parameter $k$ is computed as $k_{K12}$ in \cite{perez-melo},
which found it to have superior average performance in coefficient testing, compared to other choices for $k$.

Figure  \ref{fig1} shows empirical power for the various simulation scenarios, with power curves for all three models plotted against $\sigma^{-1}$ (here, increasing $\sigma^{-1}$ increases the effect size). Notice that  the GS model outperforms the other models for positive $\rho$ values and is underpowered for negative  $\rho$. The power difference becomes more pronounced with higher  $\rho$  and $p$ values. This is not surprising, as the  coefficient  of the first predictor, $\beta_1$ cumulates larger indirect effects from more variables, in these cases. This means that more severe multicollinearity actually helps the GS test, so long as the independent variables are positively correlated. Otherwise, testing based on ridge regression is consistently more powerful that testing under the naive model, though by a modest amount. In these situations, we can see that the metric $\Delta$ is a faithful discriminant of power between the GS and naive regressions, indicating superior performance for values greater than 1, no power for GS when $\Delta =0$, and even declining power for negative values.

\begin{figure} [htb]
\centering
\includegraphics[width=1\linewidth]{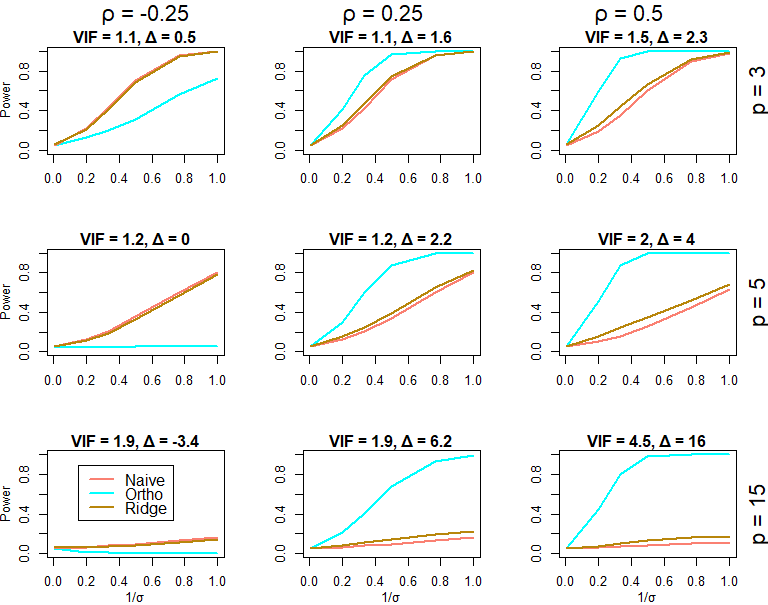}
\caption{Power profiles for the first coefficient t-test under the naive, Gram--Schmidt, and ridge regression models, for different values of $\rho, p$ and $\sigma$. All tests are one-sided. The variance inflation factor and average $\Delta$ are shown, for each setting.}
\label{fig1}
\end{figure}

\subsection{Simulation study of power in causal mediation} \label{study-med}
To demonstrate the advantage of the GS approach when testing causal effects over standard mediation methods, we simulate a causal model with a single mediator. The parent variable $X$ acts on outcome   $Y=\rho M + 0.4X + \epsilon$ via mediator $M=0.25X + Z$, where $X, Z$, and $\epsilon$ are independent standard normal. 

Type I error studies of size $n=100$ are generated by setting $\rho=0$, and for power studies we set $\rho=0.7$. Mediated (or indirect) effect sizes are controlled by varying $\rho$ from 0 (no mediation) to 500 (effectively 100\% of the total effect of $X$ is mediated). Figure \ref{fig-med} shows the empirical power curve ($N=1000$ replications) for one--sided testing of the indirect effect and fraction mediated, using the GS method versus the standard tests implemented in the  \textit{lavaan} R package. We compare against both the delta method and the bootstrap implementations in \textit{lavaan}. Notice first that type I error (at effect size $\rho=0$) is closer to nominal using our vector geometric approach, compared to  \textit{lavaan}, which is too conservative. Power is superior for our approach throughout. Interestingly, tests for indirect effects approach an asymptote that is strictly smaller than one (though very close to 1 for GS). This is because the indirect effect is given by the product $0.25\rho$, and while we can make $\rho$ very large, the estimated effect of $X$ on $M$ still has a chance that it might be null, independently of $\rho$. The proportion mediated (right panel of Figure  \ref{fig-med}) is affected by the
bootstrap method`s frequent failures to converge, resulting in an apparent loss of power compared to the delta method.
For larger sample sizes ($n=500$), performance was found to be identical for GS and \textit{lavaan}, but for smaller samples ($n=50$) the differences became larger (results not shown). This suggests that the GS approach should be preferred  over  \textit{lavaan}  for (relatively) small sample sizes.

\begin{figure} [htb]
\centering
  \begin{subfigure}[t]{0.48\textwidth}
        \centering
        \includegraphics[width=\linewidth]{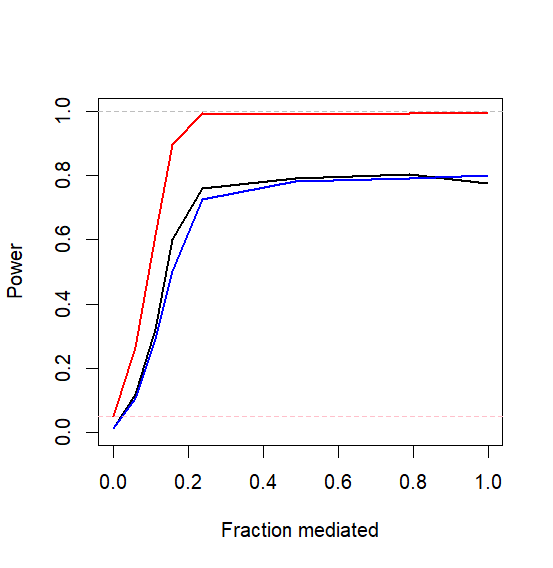}
    \end{subfigure}
  \begin{subfigure}[t]{0.48\textwidth}
        \centering
        \includegraphics[width=\linewidth]{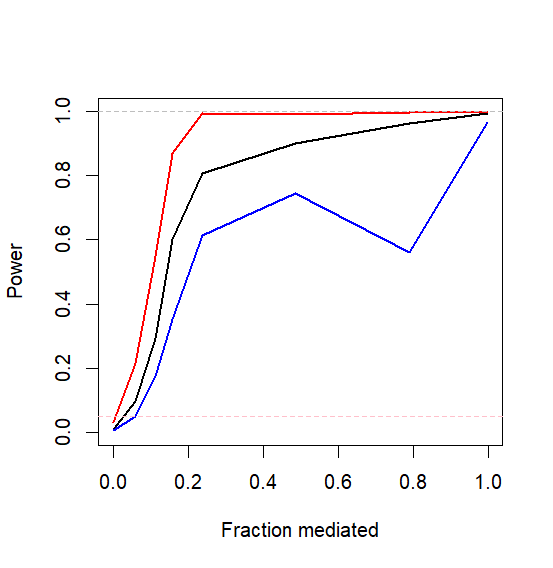}
    \end{subfigure}
\caption{Power profiles of tests for positive indirect effect (\textit{left}) and fraction mediated (\textit{right}), for the delta method (black), Gram--Schmidt (red), and bootstrap (blue). Effect sizes are measured as the fraction of total effect mediated.  All tests are one-sided. The nominal significance level of 0.05 is marked by a dashed line.}
\label{fig-med}
\end{figure}

\subsection{Lag determination in a causal feedback loop} \label{study-loop}
Suppose we have a feedback loop between two variables, $V$ and $W$, observed at times $t=1,...,T$, where $T=100$. The generating process for the two variables is governed by the recursions
\begin{align*}
V_{t+1} &= \sin(t) W_{t-1} + \epsilon_{1,t} \\
W_{t+1} &= \cos(t) V_{t-2} +  0.5 W_t +\epsilon_{2,t},
\end{align*}
where $ \epsilon_{i,t}$ are all independent standard normal variates. These two longitudinal variables are  measured on each of $n=250$ individuals.  The lags can be seen in the causal diagram of this loop in Figure \ref{fig-dag}. The causal problem here is to determine the correct lags from observed data pairs $(V_t, W_t)$.  Notice that the use of $\sin(t)$ and $\cos(t)$ ensures that the system is non--stationary; it is easy to check that $Cov(V_t, W_{t+3})$ depends on $t$. Thus, using the CCF for lag determination does not make sense.
Taking the outcome to be $W_T$, we use the GS approach on predictors $V_1, W_1, V_2, W_2, ...$ (in this order) and compute the fraction of total effect of all variables in $V_{+0}$ and $W_{+0}$ mediated by variables in the right--shifted sets $V_{+d}$ and $W_{+d}$, for $d=0,1,...,7$. We use both measures $D^1$ and $D^2$, and compare their performance. We run a total of 1000 studies and, for each measure $D^1_{V,W,d}, D^1_{V,V,d}, D^1_{W,V,d}$, and $D^1_{W,W,d}$ and lag $d$ we report power as the fraction of studies where lag $d$ was found to be significant. Similarly for the  $D^2$ measure.

Results in Figure \ref{fig-loop} show that this method correctly identifies the nominal lags used in generating the feedback loop: 3 time units for changes in $V_t$ to reach  $W_t$, 1 unit corresponding to the causation $W_t \rightarrow W_{t+1}$ and 2 units from   $W_t$ to reach  $V_t$. Although there are no direct arrows from $V$ to $V$, the algorithm detects a lag of 5, corresponding to the compound path $V_t  \rightarrow W_{t+3}  \rightarrow V_{t+3+2}$. Lags beyond the smallest direct lag also show up as significant due to the various combinations of paths connecting them. As expected, testing measure $D^2$ is more powerful compared to testing $D^1$ because  $D^2$ contains vector information, while  $D^1$ is susceptible to inconsistent mediation as it sums up mediated effects of potentially different signs. 

\begin{figure} [htb]
\centering
\includegraphics[width=1\linewidth]{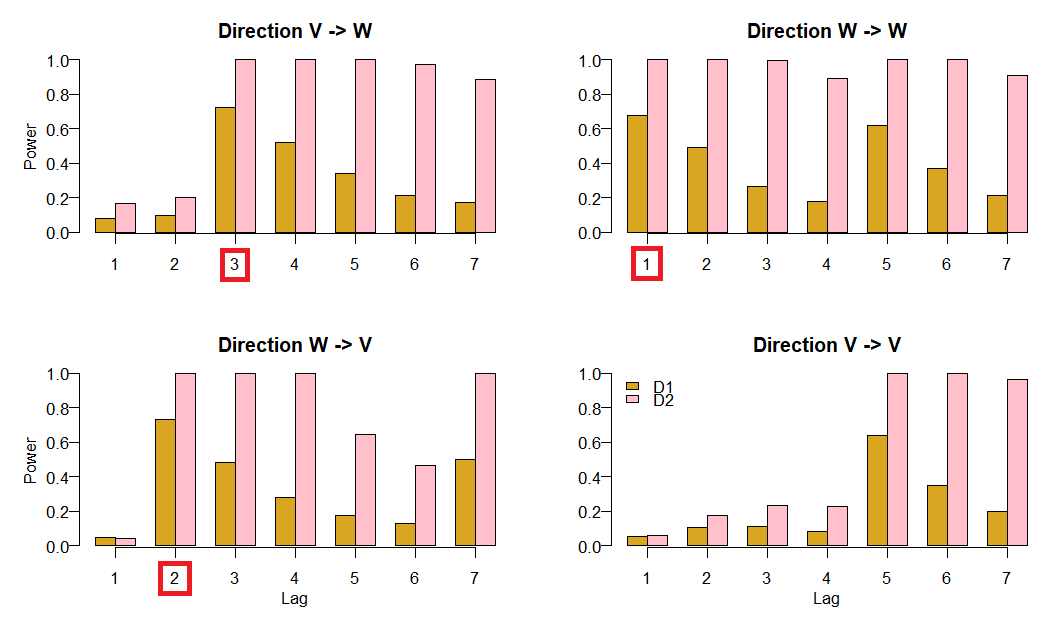}
\caption{ Power to detect lags in a feedback loop between time series $V_t$ and $W_t$ by testing the fraction of indirect effect mediated, e.g., of  $V_t$ by $W_{t+d}$ for lags $d=1,2,...,7$. Directions are from originating to mediating variable sets. Lags boxed in red are the ones assumed in the data generating process. }
\label{fig-loop}
\end{figure}

\subsection{Computing a follow-up sample size for the Demand for HIV testing study in Malawi using exact distributions}
A %high-profile 
 study published in the American Economic Review (\cite{thornton}) investigated whether knowledge of one's HIV status causes individuals to change their sexual behavior in response. The aim was to gauge the effectiveness of offering free testing and counselling services for combatting the spread of HIV. The study comprised about 2,800 individuals from rural Malawi who have been tested for HIV at baseline. Intervention was in the form of offering vouchers with randomized monetary value to respondents to incentivize them to learn their test results. At the end of the study, participants were questioned about their sexual behavior, including whether they had purchased any condoms, how many, and whether they had sex in the previous two months at follow--up. The findings of the study were mixed: while they found a statistically significant number of participants buying condoms at follow-up, the number of condoms bought was low, and there was insufficient evidence of any change in reported sexual behavior at follow--up. %Given that this study is framed and analyzed in causal terms, and 
Given that the direction of their non--significant results aligned with their positive findings, there is a case to be made that perhaps the sample size was insufficient.  Here, we will use our methods to compute the sample size required to attain significance  for the effects of their  policy intervention. 

An analysis similar to that in the original paper is given in Table \ref{HIV}: the outcome  `reported having sex at follow-up' is regressed via OLS (not logistic) on the intervention (`got' test results), HIV status and other covariates, for individuals who reported having sex at baseline ($n=1,011$).  %Assuming that HIV testing is error--free, we are powering a future study at 80\% to demonstrate that the intervention (getting results) significantly reduces reports of having sex at follow-up, given that the effect sizes obtained in the existing study are accurate and do not change over time. %A proposed causal diagram for this experiment is given in Figure... 
Causal precedence of variables is relatively straightforward in this case: some variables are determined at birth (age, male, region -- arguably), HIV later in life, and getting test results is the most recent. Distance to testing centers was randomized so its order of inclusion shouldn't matter; here, we include it with variables at birth. %we computed its correlation with other predictors, most of which were low, as expected; its only high correlation was with location (Rumphi, at $-0.415$), suggesting a difference in mobility at that location. 
%we include it before location in the temporal order, which is: male, age, age squared, distance to test location, region, HIV status, and whether the participant got the results or not (`Got'). 
We keep the interaction term $got \times HIV$ and place it last. This is important because individuals who got positive HIV results are likely to behave differently from those who got negative results. %We can  estimate the value of the intervention in a couple of ways.  Perhaps the most straightforward is to include the interaction term as a separate new term which comes after the two main effect terms. 
The total effect of the policy in our approach is $TE(Got^\perp) +TE((Got \times HIV)^\perp)$, where $TE$ denotes the total effect, computed in Proposition \ref{causal}. 
Another approach to estimate the value of HIV tests is as a mediator for HIV status itself. This status can be seen as the most important state variable, because it measures indirectly risky behavior, as well as participant's expectation as to whether they carry HIV or not, in the absence of testing. Table \ref{HIV} computes the indirect effect and proportion of HIV mediated via testing, as an alternative measure of this policy. Notice that all effect estimates agree between lavaan and GS, and standard errors are slightly smaller with GS. %Notice especially how close the CI for fraction mediated is for GS compared to the bootstrapped version. 
The proportion mediated is known to be a noisy estimate, especially due to the denominator being statistically no different from zero, which gives the ratio heavy tails. Still, GS has the computational advantage of producing an exact CI, whereas the bootstrapped CI is highly random. Finally, we compute the sample size required for a follow--up study, powered at 80\% to show a significant total policy effect, i.e., that testing for HIV reduces reports of having sex at follow-up. For this, we use scaled--up versions of $\bm{\beta}$ and $Q$  for a study $k$ times larger (namely, $\bm{\beta}^{(k)}=\sqrt{k}\bm{\beta}$ and  $Q^{(k)}=\sqrt{k}Q$, see proof of Theorem \ref{thm5}), and solve for $k$. The calculation is trivial due to independence among the components of $\bm{\beta}$. We then repeat the exercise for a significant fraction mediated ($W$) using the exact null and one--sided alternative distributions in the Appendix.
From table \ref{HIV}, a study about 18 times larger than the original  would be necessary to show a significant policy total effect, and about 13 times larger to show that the intervention mediates a significant fraction of HIV status effect. 

\begin{table}[]
\begin{tabular}{llll|lll|lll}
\hline
\multirow{3}{*}{OLS}  & Intercept & male    & age     & age$^2$    & dist             & rumphi          & HIV          & got              & got$\times$HIV \\ \hline

                      & 0.197       & -0.020  & 0.027   & 0.000   & 0.027            & -0.067          & -0.028           & 0.018            & -0.087    \\
                      & (0.107)     & (0.028) & (0.005) & (0.000) & (0.013)          & (0.031)         & (0.095)          & (0.032)          & (0.123)   \\ \hline \hline
                      &             &         &         &         & lavaan           &                 &                  & GS                                 &           \\
\multicolumn{4}{l|}{Total effect of (got + got$\times$HIV)}          &         & -0.074           &                 &                  & -0.074                          &           \\
                      &             &         &         &         & (0.136)          &                 &                  & (0.126)                         &           \\
\multicolumn{3}{l}{Indirect effect of HIV}    &         &         & -0.053           &                 &                  & -0.053                           &           \\
                      &             &         &         &         & (0.081)          &                 &                  & (0.074)                            &           \\
\multicolumn{3}{l}{Fraction of HIV mediated}      &         &         & 0.660            &                 &                  & 0.660                              &           \\
                      &             &         &         &          \multicolumn{3}{c|}{(-9.242, 8.364)} & \multicolumn{3}{c}{(-6.815,   8.865)}            \\\hline
\multicolumn{4}{l|}{\multirow{2}{*}{Sample   size of follow-up:}} &         &                  &                 &           $k$       &          &  $n_k$        \\
 &           \multicolumn{3}{l|}{  \multirow{2}{*}{for  TE(got + got$\times$HIV)}}                   &         &     --   &                 &             18.016       &    & 18,214       \\ 
 &           \multicolumn{3}{l|}{  \multirow{2}{*}{for frac of HIV mediated}}                   &         &     --   &                 &          13.228         &                &      13,373     \\ 
 							&		&         &         &         &        &                 &         &                  &           \\\hline
\end{tabular}\caption{Analysis of the demand for HIV testing data: the OLS results report direct effects (\textit{top}); different measures of policy effectiveness (\textit{middle}) are computed using \textit{lavaan} and GS analyses,  assuming the temporal order of variables in the top row; (\textit{bottom}) sample size for a follow-up study to attain a positive result with 80\% power. SEs (for normal estimators) and CIs (for the non-normal ratio) are given in parantheses. }  \label{HIV}
\end{table}

\section{Discussion} \label{disc}

This paper makes a couple of advances that hold significant potential for statistical practice. In mediation analysis, it offers a granular framework of effect transmission through a DAG, along with exact distributions that can be used to estimate and plan studies targeting a variety of causal effects. This resolves a longstanding weakness in the field, which relies on the product rule and path counting to compute indirect effects. Notably, the proposed framework considerably simplifies sample size calculations for mediation, which would otherwise be done via  simulation methods. Our approach relies on the existance of a temporal DAG, which is a broad class of path diagrams based on the strong intuitive support of time separation. To correctly identify a DAG, though, causal discovery methods are sometimes needed, and the transmission of indirect effects holds potential for testing in this area, as was illustrated with pruning DAGs corresponding to feedback loops.  I plan to further explore this framework's application to this and other causal topics, such as computing natural direct/indirect effects and counterfactual reasoning.
%Although many other methods exist... opens the door to further investigating their use.

Another limitation in the present paper is that errors are normally distributed. Although this is often not a concern in practice, where normality tends to be assumed by default, the GS algorithm itself and the vector logic of causal effects do not inherently assume the Gaussian distribution. Thus, it should be possible to adapt most of the results here to work with skewed or otherwise non--normal distributions.

In the more general regression analysis space,  we showed that path diagrams are  a natural conceptual expansion to multiple regression because they can help explain away multicollinearity and turn it into a source of power instead of an inconvenience to be mitigated. 
This contribution is relevant and timely, both due to the near--universal use of regression in practice, but also to increasing demand on the part of applied researchers seeking statistical consulting services to investigate mediated and causal relationships in addition to classical association results. This is entirely justified because science often demands a more specific understanding of relationships than simply establishing correlation. By proposing that total effects be testing instead of direct effects, the current framework addresses both the problem of power/multicollinearity, which is encountered in both causal and association analyses, but also provides a blueprint of how one can turn a classical regression into a mediation problem, in a fairly general case. Finally, as a technical point, the Delta metric introduced here is poised to replace VIF as an indicator of multicollinearity, especially when one is looking at whether a mediation analyisis holds better potential for power compared to multiple regression. Due to the widespread use of VIF, this is a significant enough advance in its own right.

%%%%%%%%%%%%%%

\begin{appendix}

\section{Proof of Lemma \ref{thm1}}

\begin{proof} 
As is typical when looking for a UMP test in the presence of nuisance parameters, we first wish to identify sufficient statistics for this inference. With normal data, the joint density will belong to the exponential family and can be written thus (here, $\bm{x}_{q,i}$ is the $i$--th component of vector $\bm{x}_q$)
\begin{align} \label{joint_lik}
f(\mathbf{Y}|\bm{\beta}, \sigma) &=\prod_{i=1}^{n}{ \frac{1}{\sqrt{2\pi}\sigma}\exp \left\{-\frac{(Y_i - \sum_{q=1}^{p}{\beta_q \bm{x}_{q,i}})^2}{2\sigma^2} \right\} } \nonumber \\
&= \frac{1}{(\sqrt{2\pi})^n\sigma^n}\exp \left\{ -\frac{\sum_{i=1}^{n}{Y_i^2}}{2\sigma^2} - \frac{\sum_{i=1}^{n}{(\sum_{q=1}^{p}{\beta_q \bm{x}_{q,i}})^2}}{2\sigma^2} +\frac{\sum_{i=1}^{n}{(Y_i \sum_{q=1}^{p}{\beta_q \bm{x}_{q,i}})}}{\sigma^2} \right\}  \nonumber \\
&= h(\bm{\beta}, \sigma) \exp \left\{ -\frac{\mathbf{Y}^\intercal \mathbf{Y}}{2\sigma^2} + \frac{\beta_1}{\sigma^2} \mathbf{x}_1^\intercal \mathbf{Y} + \frac{\beta_2}{\sigma^2} \mathbf{x}_2^\intercal \mathbf{Y} +...+ \frac{\beta_p}{\sigma^2} \mathbf{x}_p^\intercal \mathbf{Y}\right\}
\end{align}
From this, the sufficient statistics are $(\mathbf{Y}^\intercal \mathbf{Y}, \mathbf{x}_1^\intercal \mathbf{Y}, \mathbf{x}_2^\intercal \mathbf{Y},..., \mathbf{x}_p^\intercal \mathbf{Y}   )$ corresponding to the natural parameter vector $(-\frac{1}{2\sigma^2}, \frac{\beta_1}{\sigma^2},...,\frac{\beta_p}{\sigma^2})$. According to \cite{bhattacharya} (pp. 147-148) there exists an unbiased UMP test $\phi_1(u,\bm{t}) = I\{u \geq c_1(\bm{t})\}$ where $c_1(\bm{t})$ is determined from $E_{\beta_p=0}[\phi_1(U, \bm{T})|\bm{T=t}]=\alpha$, where $U,\bm{T}$ are the sufficient statistics for the important and nuisance parameters, respectively. The problem is that the joint conditional distribution $(U, \bm{T})|\bm{T=t}$ is not yet straightforward to obtain as $U= \mathbf{x}_p^\intercal \mathbf{Y} $ is not entirely independent of  $\bm{T}$. In what follows, the plan is to use Theorem 6.9.2 part A from \cite{bhattacharya}, which gives some relatively simpler conditions for a test to attain UMPU property, and is especially suited when data is normal.

Our objective now is to find a simpler characterization for the distribution of the sufficient statistics. Following and extending the reasoning in the aforementioned Example 6.9.11, let $\mathbf{a_1, a_2,...,a_n}$ be an orthonormal basis for $\mathbb{R}^n$ that includes the covariate vectors, i.e., $\mathbf{a}_1=\mathbf{x}_1, \mathbf{a}_2=\mathbf{x}_2, ...,\mathbf{a}_p=\mathbf{x}_p $; the other vectors $\mathbf{a}_{p+1},...,\mathbf{a}_n$ are chosen such that $\mathbf{a}_i^\intercal \mathbf{a}_i = 1$, for all $i$ from $p+1$ to $n$, and $\mathbf{a}_i^\intercal \mathbf{a}_j = 0$ when $i \neq j$. Further define $W_i = \mathbf{a}_i^\intercal \mathbf{\epsilon}, \forall i$. It is relatively straightforward to show that $W_1,W_2,...,W_n$ are iid $N(0,\sigma^2)$. We also have that $\sum_{i}{W_i^2}=\sum_{i}{\epsilon_i^2}$. This is true because  
$W_i$ is the length of the projection of the error vector $\bm{\epsilon}$ on basis vector  $\mathbf{a_i}$, and we express the squared length of vector $\mathbf{\epsilon}$ in both coordinate bases.

In the regression model, we can identify the best fit parameters $\beta_i, i=1,..,p$ as the projection of data vector $\mathbf{Y}$ onto covariate directions $\mathbf{x}_i=\mathbf{a}_i$. Let us call the corresponding estimators $B_i=\mathbf{a}_i^\intercal \mathbf{Y}=\mathbf{a}_i^\intercal (\beta_1 \mathbf{a}_1+...\beta_p \mathbf{a}_p+\mathbf{\epsilon})=\beta_i+W_i$. The residual sum of squares is $R=\sum_{i=p+1}^{n}{W_i^2}=\sum_{i=1}^{n}{\epsilon_i^2} - \sum_{i=1}^{p}{W_i^2}=\sum_{i=1}^{n}{(Y_i-\beta_1 a_{1,i}-\beta_2 a_{2,i}-...-\beta_p a_{p,i})^2} -\sum_{i=1}^{p}{(B_i -\beta_i)^2}$. The first sum expands to $\mathbf{Y}^\intercal \mathbf{Y}-2\sum_{i=1}^{p}{\beta_i \mathbf{a}_i^\intercal \mathbf{Y}} + \sum_{i=1}^{p}{\beta_i^2}$. It is then easy to obtain that  $R=\mathbf{Y}^\intercal \mathbf{Y}-\sum_{i=1}^{p}{B_i^2} \sim \sigma^2\chi_{n-p}^2$, from the original definition of $R=\sum_{i=p+1}^{n}{W_i^2}$.

To recapitulate, we found summary statistics $B_i \sim N(\beta_i, \sigma^2), i=1,..,p$, and $R$, which are all mutually independent. Plugging these into Equation \ref{joint_lik}, we have
\begin{equation}
f(\mathbf{Y}|\bm{\beta}, \sigma)= h(\bm{\beta}, \sigma) \exp \left\{ -\frac{R+ \sum_{i=1}^{p}{B_i^2}}{2\sigma^2} + \frac{\beta_1}{\sigma^2}B_1 + \frac{\beta_2}{\sigma^2} B_2 +...+ \frac{\beta_p}{\sigma^2} B_p \right\}.
\end{equation}
From this, we see that statistics $(U,T_1,...,T_p):=(B_p, B_1, ...,B_{p-1},R+ \sum_{i=1}^{p}{B_i^2})$ are  sufficient for $(\frac{\beta_p}{2\sigma^2}, \frac{\beta_1}{2\sigma^2},...,\frac{\beta_{p-1}}{2\sigma^2}, -\frac{1}{2\sigma^2})$. Next, define a new variable $V=g(U,T_1,T_2,...,T_p)$ as
\[
V=\frac{U}{\sqrt{\frac{T_p-T_1^2-T_2^2-...-T_{p-1}^2-U^2}{n-p}}} =\frac{B_p}{\sqrt{\frac{R}{n-p}}},
\]
and check that $V$ satisfies the conditions of Theorem 6.9.2, namely
\begin{enumerate}
\item $V$ is independent of $\bm{T}=(T_1,...,T_p)$ when $\beta_p/\sigma^2=0$. As $B_p \sim N(0,\sigma^2)$ when $\beta_p=0$, we get $V \sim t_{n-p} $. As the distribution of $V$ does not depend on any of the other parameters $(-\frac{1}{2\sigma^2}, \frac{\beta_1}{2\sigma^2},...,\frac{\beta_{p-1}}{2\sigma^2})$, it follows from Corollary 5.1.1 to Basu's Theorem in \cite{lehmann} that $V$ is independent of $\bm{T}$.
\item $g(u,\bm{t})$ is increasing in $u$ for each $\bm{t}$. It is easy to show $\frac{\partial g}{\partial u}>0$ for any value of  $\bm{t}$.
\end{enumerate}

Therefore, we can conclude that an UMP unbiased test for $\beta_p/\sigma^2 \leq 0$ vs $\beta_p/\sigma^2 > 0$, which is equivalent to testing $H_0$ vs $H_1$ is 
\[
\phi(v)=\begin{cases}
0, \quad  \text{if $v<c$} \\
\text{$\xi$}, \quad  \text{if $v=c$} \\
1, \quad  \text{if $v> c$},
\end{cases}
\]
where $c$ and $\xi$ are determined by $E_{\beta_p = 0}[\phi(V)]=\alpha$. Ignoring the middle case ($V=c$) which has probability zero, this means $P_{\beta_p = 0}(V>c)=\alpha$, i.e., $c=t_{n-p,1-\alpha}. $
\end{proof}

\section{Proof of Theorem \ref{thm2}}

\begin{proof}  We follow the same proof as in Lemma \ref{thm1} by constructing the GS decomposition of the design matrix M  (assuming the first column holds the predictor of interest) which leads us to reparameterize the original model
\begin{align*}
\bm{Y} &=\alpha_1 \mathbf{m}_{1}+\alpha_2 \mathbf{m}_{2}+...+\alpha_p \mathbf{m}_{p}+ \bm{\epsilon} \qquad  \tag{A}
\text{as}\\
\bm{Y} &=\beta_1 \mathbf{x}_{1}+\beta_2 \mathbf{x}_{2}+...+\beta_p \mathbf{x}_{p}+ \bm{\epsilon} \tag{B}
\end{align*}
where $\bm{\epsilon} \sim N(0,\sigma^2 I)$ and  $M=XQ$, with $X$ orthonormal, and $Q$ upper triangular. To see the connection between the two sets of parameters, write model $A$ as  
\begin{equation} \label{model:alpha}
\bm{Y}= (\bm{m}_1, \bm{m}_2, \bm{m}_3,..., \bm{m}_p) \bm{\alpha}+ \bm{\epsilon}= XQ \bm{\alpha} + \bm{\epsilon}.
\end{equation}
Putting $\bm{\beta}=Q\bm{\alpha}$ we see this to be equivalent to model  $B$, which is written in terms of parameters $\bm{\beta}$.
The ordinary least squares estimate for  $\bm{\alpha}$ is 
\begin{align*}
\hat{\bm{\alpha}} &=[(XQ)^\intercal XQ]^{-1} (XQ)^\intercal \bm{Y} =[Q^\intercal (X^\intercal X)Q]^{-1}Q^\intercal X^\intercal \bm{Y}\\
&=Q^{-1}(Q^\intercal)^{-1} Q^\intercal X^\intercal \bm{Y} =Q^{-1} X^\intercal \bm{Y}=Q^{-1} \hat{\bm{\beta}},
\end{align*}
where we have used that the $p \times p$ matrix $Q$ is full rank.

%WLOG, let us consider the first variables to be of interest. 
Writing the likelihood for A in a similar way as (\ref{joint_lik}), we have
\begin{align} \label{jjjjj}
f(\mathbf{Y}|\bm{\alpha}, \sigma) &= \prod_{i=1}^{n}{ \frac{1}{\sqrt{2\pi}\sigma}\exp \left\{-\frac{(Y_i - \sum_{q=1}^{p}{\alpha_q \bm{m}_{q,i}})^2}{2\sigma^2} \right\} } %=\prod_{i=1}^{n}{ \frac{1}{\sqrt{2\pi}\sigma}\exp \left\{-\frac{(Y_i - \sum_{q=1}^{p}{\beta_q \bm{x}_{q,i}})^2}{2\sigma^2} \right\} } \nonumber \\
&= h(\bm{\alpha}, \sigma) \exp \left\{ -\frac{\mathbf{Y}^\intercal \mathbf{Y}}{2\sigma^2} + \frac{1}{\sigma^2} \bm{\alpha}^\intercal M^\intercal \mathbf{Y} \right\}.
\end{align}
Call the sufficient statistics $(U, T_1, ...,T_p)=( \mathbf{m}_1^\intercal \mathbf{Y} , \mathbf{m}_2^\intercal \mathbf{Y},..., \mathbf{m}_{p}^\intercal \mathbf{Y} , \mathbf{Y}^\intercal \mathbf{Y}  )$ corresponding to the natural parameter vector $(-\frac{1}{2\sigma^2}, \frac{\alpha_1}{\sigma^2},...,\frac{\alpha_p}{\sigma^2})$. Define 
\[
V= \frac{\hat{\alpha}_1}{\sqrt{\frac{\mathbf{Y}^\intercal \mathbf{Y} - ||M \hat{\bm{\alpha}}||^2}{n-p}}} =\frac{(M^\intercal M)^{-1}_{1*}M^\intercal \mathbf{Y}}{\sqrt{\frac{\mathbf{Y}^\intercal \mathbf{Y} - ||M (M^\intercal M)^{-1}M^\intercal \mathbf{Y}||^2}{n-p}}}.
\]
Putting $M^\intercal \mathbf{Y} = (U, T_1, ..., T_{p-1})^\intercal$, it is easy to see that $V$ is a function of  $(U, T_1, ...,T_p)$. To find the distribution of $V$, notice that the numerator can be written as 
\[
\hat{\alpha}_1=\bm{q}_1^\intercal \hat{\bm{\beta}}= \bm{q}_1^\intercal (B_1,...,B_p)^\intercal \sim N(\bm{q}_1^\intercal \bm{\beta}, \frac{\sigma^2}{n-p}||\bm{q}_1||^2),
\]
 according to the proof of  Lemma \ref{thm1}. As well, the denominator is still $\sqrt{\frac{R}{n-p}}=\sqrt{\frac{\sigma^2 \chi^2_{n-p}}{n-p}} $, which makes 
\[
V \sim \frac{||\bm{q}_1||}{\sqrt{n-p}} t_{n-p}
\]
at the boundary point $\alpha_1 = \bm{q}_1^\intercal \bm{\beta} = 0$. Thus, the distribution of $V$ is independent of the nuisance parameters.

Secondly, to show that $V$ is an increasing function of $U$ for each $\bm{T}$, we show how the numerator and  denominator depend on $U$. % OR, reparameterize as   $(U, T_1, ...,T_p)=( \mathbf{m}_1^\intercal \mathbf{Y} , \mathbf{m}_2^\intercal \mathbf{Y}, \mathbf{m}_2^\intercal \mathbf{Y},..., \mathbf{m}_{p-1}^\intercal \mathbf{Y} , \mathbf{Y}^\intercal \mathbf{Y} - ||M (M^\intercal M)^{-1}M^\intercal \mathbf{Y}||^2 )$ corresponding to the natural parameter vector $(-\frac{1}{2\sigma^2}, \frac{\alpha_1 + what}{\sigma^2},...,\frac{\alpha_p + what}{\sigma^2})$.
For the numerator, we have
\begin{align*}
\hat{\alpha}_1 &= [(M^\intercal M)^{-1}]_{1*} M^\intercal \mathbf{Y}=[Q^{-1}(Q^{-1})^\intercal]_{1*} M^\intercal \mathbf{Y}=\bm{q}_1^\intercal (Q^{-1})^\intercal %\begin{pmatrix} 
% U \\ 
% T_1 \\ 
% \vdots \\
% T_{p-1}
%\end{pmatrix}  \\
(U, T_1, ..., T_{p-1})^\intercal \\
&= ||\bm{q}_1||^2 U + T_1 \bm{q}_1\cdot \bm{q}_2 + ...+T_{p-1} \bm{q}_1\cdot \bm{q}_p = ||\bm{q}_1||^2 U +  \bm{q}_1\cdot \bm{w}_{\bm{T}},
\end{align*}
where we have defined $\bm{w}_{\bm{T}}=T_1  \bm{q}_2 + ...+T_{p-1} \bm{q}_p$.
Next, we can write the $SSE$ in the numerator as
\begin{align*}
SSE &= \mathbf{Y}^\intercal \mathbf{Y} - ||M (M^\intercal M)^{-1}M^\intercal \mathbf{Y}||^2 = \mathbf{Y}^\intercal \mathbf{Y} - \mathbf{Y}^\intercal M (M^\intercal M)^{-1}M^\intercal M (M^\intercal M)^{-1}M^\intercal \mathbf{Y} \\
&= \mathbf{Y}^\intercal \mathbf{Y} - \mathbf{Y}^\intercal M (M^\intercal M)^{-1}M^\intercal \mathbf{Y} =
\mathbf{Y}^\intercal \mathbf{Y} - (U, T_1, ...,T_{p-1}) Q^{-1}(Q^{-1})^\intercal
\begin{pmatrix} 
 U \\ 
 T_1 \\ 
 \vdots \\
 T_{p-1}
\end{pmatrix}  \\
&=  \mathbf{Y}^\intercal \mathbf{Y} - (U \bm{q}_1^\intercal + T_1 \bm{q}_2^\intercal +...+ T_{p-1} \bm{q}_p^\intercal)  (U \bm{q}_1 + T_1 \bm{q}_2 +...+ T_{p-1} \bm{q}_p) \\
&=  \mathbf{Y}^\intercal \mathbf{Y} - U^2 ||\bm{q}_1||^2 - 2 U \bm{q}_1 \cdot   \bm{w}_{\bm{T}} - ||  \bm{w}_{\bm{T}}||^2.
\end{align*}

Thus, $V= \frac{||\bm{q}_1||^2 U +  \bm{q}_1\cdot \bm{w}_{\bm{T}}}{\sqrt{T_p - U^2 ||\bm{q}_1||^2 - 2 U \bm{q}_1 \cdot   \bm{w}_{\bm{T}} - ||  \bm{w}_{\bm{T}}||^2}} \sqrt{n-p}.$ Using the following shorthand notations:  $a=||\bm{q}_1||^2,\ b= \bm{q}_1 \cdot   \bm{w}_{\bm{T}},$ and $d= ||  \bm{w}_{\bm{T}}||^2$  we can check the sign of the partial derivative, assuming $T_i$ as constant:
\begin{align*}
& \frac{\partial V}{\partial U}=  \sqrt{n-p} \frac{a(T_p - a U^2 - 2b U  - d)^{1/2} - \frac{1}{2} (a U + b)(T_p - a U^2 - 2b U  - d)^{-1/2} (-2 a U -2 b) }{T_p - a U^2 - 2b U  - d} > 0\\
\iff & a(T_p - a U^2 - 2b U  - d)^{1/2}  >\frac{1}{2} (a U + b)(T_p - a U^2 - 2b U  - d)^{-1/2} (-2 a U - 2b)\\
\iff & 2a(T_p - a U^2 -2 b U  - d) > -2a^2 U^2 - 2abU - 2abU- 2b^2\\
\iff & 2aT_p - 4ab U  - 2ad > -  4abU- 2b^2 \iff 2aT_p - 2ad +2b^2 > 0.
\end{align*}
This means
\begin{equation} \label{ineq}
T_p ||\bm{q}_1||^2 -  ||\bm{q}_1||^2 ||\bm{w}_{\bm{T}}||^2 +  \bm{q}_1^\intercal \bm{w}_{\bm{T}} \bm{q}_1^\intercal \bm{w}_{\bm{T}} >0 \iff T_p  -  ||\bm{w}_{\bm{T}}||^2 + \frac{(\bm{q}_1^\intercal \bm{w}_{\bm{T}} )^2}{||\bm{q}_1||^2} > 0. 
\end{equation}
From the expression for $SSE$ above, we can write $T_p  -  ||\bm{w}_{\bm{T}}||^2 =SSE + U^2 ||\bm{q}_1||^2 + 2 U \bm{q}_1 \cdot   \bm{w}_{\bm{T}}$. Substituting this into \ref{ineq}, the condition becomes
\begin{align*}
& SSE + U^2 ||\bm{q}_1||^2 + 2 U \bm{q}_1^\intercal   \bm{w}_{\bm{T}} +\frac{(\bm{q}_1^\intercal \bm{w}_{\bm{T}} )^2}{||\bm{q}_1||^2} > 0 \\
\iff & SSE +\left( U ||\bm{q}_1|| + \frac{\bm{q}_1^\intercal \bm{w}_{\bm{T}} }{||\bm{q}_1||} \right)^2 > 0, 
\end{align*}
which is true for any non-zero-fit model.
\end{proof}

\section{Proof of Theorem \ref{thm5}}

\begin{proof} 
\textit{Part (a)}  From the proof of Theorem \ref{thm2}, we have $\hat{\bm{\alpha}} =Q^{-1} \hat{\bm{\beta}}$.
Furthermore, using the well--known formula for the variance--covariance matrix of the OLS estimate, we have  
$$\text{Var}(\hat{\bm{\alpha}} )=\sigma^2[(XQ)^\intercal XQ]^{-1}=\sigma^2(Q^\intercal Q)^{-1}=\sigma^2 Q^{-1} (Q^{-1})^\intercal.$$
 The standard error of $\hat{\bm{\alpha}}$ is computed by replacing $\sigma^2$ with $s^2 = \frac{\hat{\bm{\epsilon}}^\intercal \hat{\bm{\epsilon}}}{n-p}=\frac{SSE}{n-p}$, which is the same for both parameterizations. Let also $Q^{-1}=(\bm{q}_1, \bm{q}_2, ...,\bm{q}_p)^\intercal$. Thus we can simplify $\hat{\alpha}_i=\bm{q}_i^\intercal \hat{\bm{\beta}}$ and $(\text{s.e.}(\hat{\alpha}_i))^2=s^2\bm{q}_i^\intercal \bm{q}_i=\frac{SSE}{n-p}||\bm{q}_i||^2$, making the t-test statistic for  $\phi_A$:
$$
 V_A =\frac{\sqrt{n-p}\, \bm{q}_i^\intercal \hat{\bm{\beta}}}{\sqrt{SSE}||\bm{q}_i||} \sim t_{n-p}. 
$$ 
The power function for $\phi_A$ in terms of the (scaled) effect sizes $\bm{\beta}$ is  
\begin{align*}
A(\bm{\beta}, \sigma) &=P_{\beta, \sigma}(V_A \geq t_{n-p,1-\alpha})= P_{\beta, \sigma} \left( \frac{ \bm{q}_i^\intercal \hat{\bm{\beta}}}{||\bm{q}_i||} \geq t_{n-p,1-\alpha} \sqrt{\frac{SSE}{n-p}} \right) \\
&=P_{\beta, \sigma} \left( \sum_{j=1}^p \frac{q_{ij} }{||\bm{q}_i||} (\beta_i + W_i) \geq t_{n-p,1-\alpha} \sqrt{\frac{SSE}{n-p}} \right)\\
&=P_{\beta, \sigma} \left(\frac{\bm{q}_{1}^\intercal \bm{\beta} }{||\bm{q}_i||} \geq t_{n-p,1-\alpha} \sqrt{\frac{SSE}{n-p}} - Z \right), 
\end{align*}
where $Z= \sum_{j=1}^p \frac{q_{ij} }{||\bm{q}_i||}  W_i$ and  quantities $W_i$ are defined in the proof of Lemma 1. Because $W_1,...,W_p \sim i.i.d. N(0, \sigma^2)$, we have $E(Z) = 0$, and $\text{Var}(Z) = \sum_{j=1}^p \frac{q_{ij}^2 }{||\bm{q}_i||^2}  \text{Var}(W_i)=\sigma^2$. Similarly, the power function for $\phi_B$ is  
\begin{align*}
B(\bm{\beta}, \sigma) &=P_{\beta, \sigma}(V_B \geq t_{n-p,1-\alpha})= P_{\beta, \sigma} \left( \hat{\beta}_i \geq t_{n-p,1-\alpha} \sqrt{\frac{SSE}{n-p}} \right) \\
&=P_{\beta, \sigma} \left( \beta_i +W_i \geq t_{n-p,1-\alpha} \sqrt{\frac{SSE}{n-p}} \right) 
=P_{\beta, \sigma} \left(\beta_i  \geq t_{n-p,1-\alpha} \sqrt{\frac{SSE}{n-p}} -W_i \right). 
\end{align*}
Since $W_i \overset{d}{=}Z$ and both $W_i$ and $Z$ are independent of $SSE$ we conclude that 
$$
B(\bm{\beta}, \sigma) > A(\bm{\beta}, \sigma) \iff \beta_i > \frac{\bm{q}_i^\intercal \bm{\beta} }{||\bm{q}_i||}.
$$

\textit{Part (b)} Assume that the initial study has true parameter vectors $\bm{\alpha}$ and $\bm{\beta}$ under models A and B, respectively. %, and that the data generating process does not change when we increase sample size by a factor of $k$. 
Denote the new design vectors as $\bm{m}_1^{(k)},\bm{m}_2^{(k)},...,\bm{m}_p^{(k)}$; each is obtained by stacking the initial vectors on top of each other $k$ times, such as $\bm{m}_1^{(k)}=(\bm{m}_1^\intercal, \bm{m}_1^\intercal,...,\bm{m}_1^\intercal)^\intercal$, and so on. This makes the design matrix of the new study $M^{(k)}=[M^\intercal \ M^\intercal ...\ M^\intercal]^\intercal$ of size $(k n_0) \times p$. 
 Similarly, denote the new orthogonal vectors as $\bm{x}_j^{(k)}$, and the coefficient vectors for the planned studies as $\bm{\alpha}^{(k)}$ and $\bm{\beta}^{(k)}$. The first thing to notice is that, while $\bm{\alpha}^{(k)}$ is the same vector as $\bm{\alpha}$ for any $k$, the scale of $\bm{\beta}^{(k)}$ changes, due to the fact vectors $\bm{x}_j^{(k)}$ are still normalized to one, making each of their components shrink. To see this for the first vector,
\begin{align*}
\bm{x}_1^{(k)} &= \frac{\bm{m}_1^{(k)}}{||\bm{m}_1^{(k)}||}= \frac{(\bm{m}_1^\intercal, \bm{m}_1^\intercal,...,\bm{m}_1^\intercal)^\intercal}{\sqrt{k \sum_{i=1}^{n_0}m_{1,i}^2}} 
=  \frac{[I_{n_0} \ I_{n_0} ... \ I_{n_0}]^\intercal \bm{m}_1}{||\bm{m}_1|| \sqrt{k} }\\ 
&=\frac{1}{ \sqrt{k}} [I_{n_0} \ I_{n_0} ... \ I_{n_0}]^\intercal \bm{x}_1 = \frac{1}{ \sqrt{k}} (\bm{x}_1^\intercal, \bm{x}_1^\intercal,...,\bm{x}_1^\intercal)^\intercal.
\end{align*}

We can follow the same argument throughout Algorithm \ref{algo}, where each residual vector $\bm{r}_j^{(k)}$ ends up being a repetition of $k$ stacked  $\bm{r}_j$ vectors from the initial study. When these residuals are normalized, we will have $\bm{x}_j^{(k)} = (\bm{x}_j^\intercal, \bm{x}_j^\intercal,...,\bm{x}_j^\intercal)^\intercal /  \sqrt{k}$. 
Solving for $Q$ from $M^{(k)}=X^{(k)}Q$ yields $Q^{(k)}=\sqrt{k}Q$, and  $\bm{\beta}^{(k)}=\sqrt{k}\bm{\beta}$. 

Next, we wish to equate power under the two models as $k \rightarrow \infty$ and establish a relationship between $n_A$ and $n_B$.
From  Lemma \ref{thm1}, we have $SSE=R \sim \sigma^2 \chi^2_{n-p}$. Hence, $\frac{SSE}{n-p} \overset{a.s.}{\rightarrow} \sigma^2$ as $n \rightarrow \infty$; as well, $t_{n-p,1-\alpha} \rightarrow z_{1-\alpha}$. Thus, the power functions for testing the $i$-th variable under models A and B become
\begin{align*}
\lim_{k_A\rightarrow \infty} A(\bm{\alpha}^{(k_A)}, \sigma) &=P_{\bm{\alpha}^{(k_A)}, \sigma} \left(\frac{\alpha_i^{(k_A)} }{||\bm{q}_i/\sqrt{k_A}||} \geq \sigma z_{1-\alpha}  - Z\right)\hspace{-0.3em},\ a.s., \text{ and} \\
\lim_{k_B\rightarrow \infty} B(\bm{\beta}^{(k_B)}, \sigma) &=P_{\bm{\beta}^{(k_B)}, \sigma} \left(\beta_i^{(k_B)}  \geq \sigma z_{1-\alpha}  -W_i\right)\hspace{-0.3em}, a.s.,
\end{align*}
where we have used the fact that the inverse of $Q^{(k)}$ is $\frac{1}{\sqrt{k}}Q^{-1}$.
As $Z, W_i \sim N(0, \sigma^2)$, equating the powers in the  limit is equivalent to the condition
\begin{align*}
\frac{\alpha_i^{(k_A)} }{||\bm{q}_i||/ \sqrt{k_A}} =\beta_i^{(k_B)} & \Leftrightarrow \frac{\alpha_i}{||\bm{q}_i||/\sqrt{k_A}} =\beta_i \sqrt{k_B}\\ 
\Leftrightarrow \sqrt{\frac{k_A}{k_B}} = \frac{\beta_i ||\bm{q}_i||}{\alpha_i} & \Leftrightarrow  \frac{n_A}{n_B} = \Delta_i^2 \text{ and } \alpha_i\beta_i >0. 
\end{align*}
\end{proof}

\subsection{Proof of Remarks 1 and 2}\label{proof-remarks}
\begin{proof}
For the first remark, it is easy to show that $Q^\intercal Q = M^\intercal M$. The squared norm $||\bm{q}_i||^2$ is the $i$-th diagonal element of the inverse of $Q^\intercal Q$, or $M^\intercal M$. But it is known that $(M^\intercal M)^{-1}$ has as its $i$-th diagonal element $\frac{1}{||\bm{m}_i||^2(1-\rho^2_{i})}$, where $\rho^2_{i}$ gives the proportion of the variance of $\bm{m}_{i}$ explained by the other covariates %[refs 3,4 from VIF Wikipedia page, pp 372,373].
\cite{faraway, rawlings}.  But $\frac{1}{1-\rho^2_{i}}$ is the definition of the VIF, which proves the claim.

Remark 2 is fairly easy to see by looking at the outcome equation in (\ref{SEM}). Due to the positivity condition, coefficient $\beta_1$ of $\bm{x}_1$ will always be higher than or equal to  $\alpha_1 \alpha_{11} = \alpha_1 ||\bm{m}_1||$, the direct effect parameter of $M_1$. Thus, $\Delta_1\geq1$ as the VIF is always at least one.  Equality happens when $\bm{m}_1$ is uncorrelated with the rest of the predictors, in which case $\Delta = VIF = 1$.
\end{proof}

\subsection{Proof of Proposition \ref{causal}}\label{proof-causal}
\begin{proof}
The effect sizes  follow directly from the discussion in Subsection \ref{interpretation}, regarding the $\bm{x}$ variables as non-random. The coefficients $\{ \alpha_{jk} \}$ are obtained without error from matrix $Q$ as $\alpha_{jk} = Q_{jk}^\intercal$.  To prove parts 2 and 3 we notice that the effects are a linear combination of the independent $\hat{\beta}_i$ estimators, making variances straightforward to compute.  For part 3, we use the fact that $\bm{q}_{i,i} Q_{ii}=1$, which is evident because both $Q$ and $Q^{-1}$ are upper triangular matrices. 

For part 4, we use the exact distribution of a ratio of uncorrelated, non-centered random normals in \cite{marsaglia}. We first write $W$ in terms of standard normal variates ($Z_1,2$), using the distributions in Parts 1---3 
\begin{align*}
W &= 1- \bm{q}_{i,i} Q_{ii} -\frac{\bm{q}_{i,-i}^\intercal \hat{\bm{\beta}_{-i}}Q_{ii}}{\hat{\beta}_i} 
=  - \frac{\bm{q}_i \bm{\beta} Q_{ii}  - \bm{q}_{i,i} \bm{\beta}_i  Q_{ii} + \sigma  Q_{ii} \sqrt{||\bm{q}_i||^2 - \bm{q}_{i,i}^2 } Z_1 }{\beta_i + \sigma Z_2}\\
&= - \sqrt{Q_{ii}^2 ||\bm{q}_i||^2-1} \quad\frac{a+Z_1}{b+Z2} ,
\end{align*}
where $a = \frac{\alpha_i Q_{i,i} - \beta_i}{\sigma  \sqrt{Q_{ii}^2 ||\bm{q}_i||^2 - 1 } }$ and $b=\frac{\beta_i}{\sigma}$. Thus, the distribution function of $W$ is given by
\begin{align*}
P(W \leq w) &= P( - \sqrt{Q_{ii}^2 ||\bm{q}_i||^2-1} \frac{a+Z_1}{b+Z2} \leq w  )\\
&= 1-P\left( \frac{a+Z_1}{b+Z2} \leq  - \frac{w}{\sqrt{Q_{ii}^2 ||\bm{q}_i||^2-1}} \right) \\
&= 1-F\left(-\frac{w}{\sqrt{Q_{ii}^2 ||\bm{q}_i||^2-1}} \right),
\end{align*}
where the CDF $F$  for a ratio of uncorrelated normals is known from \cite{marsaglia} to be
\begin{align}
F(t) &= L\left(\frac{a-bt}{\sqrt{1+t^2}}, -b, \frac{t}{\sqrt{1+t^2}} \right) + L\left(\frac{-a+bt}{\sqrt{1+t^2}}, b, \frac{t}{\sqrt{1+t^2}} \right) \\
&= \Phi\left(\frac{a-bt}{\sqrt{1+t^2}}, -b, \frac{t}{\sqrt{1+t^2}} \right) + \Phi\left(\frac{-a+bt}{\sqrt{1+t^2}}, b, \frac{t}{\sqrt{1+t^2}} \right),
\end{align}
where $L(u,v,\rho)=P[U>u, V>v]$ is the bivariate standard normal survival function and $\Phi(u,v,\rho)=P[U<u, V<v]$ is the bivariate normal CDF, for standard normal variates $U$ and $V$ with covariance $\rho$.  The latter is implemented via function \textit{pbinorm} in R package \textit{VGAM}. 

From the first expression of $W$ above, we can write it via a ratio of independent estimators, and find:
$E[W] = \frac{E[\bm{q}_{i,-i}^\intercal \hat{\bm{\beta}_{-i}}Q_{ii}]}{E[\hat{\beta}_i]} = 1 - \frac{E[\bm{q}_{i}^\intercal \hat{\bm{\beta}}Q_{ii}]}{E[\hat{\beta}_i]} =  1 - \frac{\bm{q}_{i}^\intercal \bm{\beta}Q_{ii}}{\beta_i}$. This is the theoretical fraction mediated, which proves that $W$ is unbiased.
\end{proof}

\subsection{Proof of Proposition \ref{indirect}}\label{proof-indirect}
\begin{proof}
Part 1. Predictor $M_j, j>i$ can be decomposed in terms of the basis vectors of the GS decomposition as $\bm{m}_j=...+\alpha_{ji}\bm{x}_i +...+\alpha_{jj}\bm{x}_j$. Every time  $\bm{m}_j$ is used as a regressor in system (\ref{SEM2})  it contributes vectors $\bm{x}_i$ (the mediated effect of $M_i$) and $\bm{x}_j$ (its own innovation) in the same proportions $\alpha_{ji}$ and $\alpha_{jj}$, respectively. This happens both in the definition of downstream predictors $M_{j+1}, M_{j+2},...$ or in the outcome regression (either as a direct term or via downstream variable terms). As its total contribution in terms of innovation to outcome $\bm{Y}$ (its total effect) is $\beta_j$ units of $\bm{x}_j$, it follows that its overall mediated effect of  $M_i$ is proportionally $\alpha_{ji} \beta_j / \alpha_{jj}$ units of $\bm{x}_i$. Since $\alpha_{ji}=Q_{ij}$ in Algorithm \ref{algo}, we conclude that the effect of $M_i$ mediated through $M_j$ is the vector $\beta_j Q_{ij} / Q_{jj} \bm{x}_i$, or in scalar terms $\frac{\beta_j Q_{ij}}{Q_{ii}Q_{jj}}$ for each unit change in  $M_i$.

Parts 2 \& 3. Here, flow of information through a mediating set refers only to pairwise, not total flow. Thus, $M_{u_1}$ mediates a fraction $\beta_{u_1} Q_{s_1 u_1}/Q_{u_1 u_1} \bm{x}_{s_1}$ from the total contribution of $\bm{x}_{s_1}$ to $\bm{Y}$; $M_{u_2}$ mediates a fraction of $\bm{x}_{s_2}$, and so on. Although  $M_{u_2}$ may also mediate $\bm{x}_{s_1}$, this is not considered here. Thus, there is no overlap between individual effect terms in  the sum. 
\end{proof}

\section{Null distributions of $D^1$ and  $D^2$}\label{D-null}
The null distribution of no mediated effects translates to $\alpha_{ji} = 0$, all $(i,j) \in SU$ in the GS decomposition in Algorithm \ref{algo}. Their estimates, $\hat{\alpha}_{ji} = Q_{ij}$ follow the Student--t distribution after standardization by $\hat{s.e.}(\hat{\alpha}_{ji})=\sqrt{\frac{\bm{r}_j^\intercal \bm{r}_j}{n-j}}=Q_{jj}/\sqrt{n-j}$.  Thus,
\begin{equation}
\sqrt{n-j}\, \frac{Q_{ij}}{Q_{jj}} \sim t(n-j).
\end{equation}
Using sets of independent variates $Z_{ij} \sim t(n-j)$ we can construct approximate empirical null distributions for $D^1$ and $D^2$ by simulating the random variables:
\begin{equation}
D^1_0=\frac{\sum_{(i,j)\in SU} \frac{\hat\beta_{j} Z_{i j}}{Q_{ii}\sqrt{n-j}}}  {\sum_{(i,j)\in SU} \hat\beta_{i} / Q_{ii}}
\end{equation}
\begin{equation}
D^2_0=\frac{\sum_{(i,j)\in SU}  \frac{\hat\beta_{j}^2  Z_{i j}^2}{n-j} }  {\sum_{(i,j)\in SU} \hat\beta_{i}^2 },
\end{equation}
where the other quantities are kept fixed.

\section{Null distribution of $W$}\label{W-null}
Under the null hypothesis that the proportion mediated is zero, we have $a=0$. The null distribution function is 
\begin{align*}
P(W \leq w_0) &= 1-F_{a=0}\left(-\frac{w_0}{\sqrt{Q_{ii}^2 ||\bm{q}_i||^2-1}} \right) \\
& =1 - \Phi \left(-\frac{\beta_i A_0}{\sigma }  , - \frac{\beta_i}{\sigma}, A_0\right)     - \Phi \left(\frac{\beta_i A_0}{\sigma }, \frac{\beta_i}{\sigma}, A_0 \right), 
\end{align*}
where $A_0=- \frac{w_0}{\sqrt{Q_{ii}^2 ||\bm{q}_i||^2-1 + w_0^2}}$. Replace $\beta_i$ with $\hat{\beta_i}$ and $\sigma$ with $s$ when computing p--values.

\end{appendix}

\section{Competing interests}
No competing interest is declared.

\section{Acknowledgments}
RGR is based at the George \& Fay Yee Centre for Healthcare Innovation. Support for CHI is provided by University of Manitoba, Canadian Institutes for Health Research, Province of Manitoba, and Shared Health Manitoba.  I thank Brenden Dufault for helpful discussions on mediation.

\section{Declaration of generative AI and AI-assisted technologies in the manuscript preparation process}
During the preparation of this work, the author used Gemini \cite{google} for generating the R code scripts for plotting some of the multipanel figures, as well as to draw the causal diagram in Figure \ref{fig-dag} (bottom). The author reviewed and edited the output as needed and takes full responsibility for the content of the published article.

\bibliographystyle{abbrvnat}
\bibliography{My_Library_of_References}

@article{choi,
	title = {Asymptotically uniformly most powerful tests in parametric and semiparametric models},
	volume = {24},
	issn = {0090-5364},
	url = {https://projecteuclid.org/journals/annals-of-statistics/volume-24/issue-2/Asymptotically-uniformly-most-powerful-tests-in-parametric-and-semiparametric-models/10.1214/aos/1032894469.full},
	doi = {10.1214/aos/1032894469},
	language = {en},
	number = {2},
	urldate = {2025-07-04},
	journal = {The Annals of Statistics},
	author = {Choi, Sungsub and Hall, W. J. and Schick, Anton},
	month = apr,
	year = {1996},
}

@article{clyde,
	title = {Prediction {Via} {Orthogonalized} {Model} {Mixing}},
	language = {en},
	author = {Clyde, Merlise and Desimone, Heather and Parmigiani, Giovanni},
	year = {2024},
}

@article{forina,
	title = {Stepwise orthogonalization of predictors in classification and regression techniques: {An} “old” technique revisited},
	volume = {87},
	copyright = {https://www.elsevier.com/tdm/userlicense/1.0/},
	issn = {01697439},
	shorttitle = {Stepwise orthogonalization of predictors in classification and regression techniques},
	url = {https://linkinghub.elsevier.com/retrieve/pii/S0169743907000469},
	doi = {10.1016/j.chemolab.2007.03.003},
	language = {en},
	number = {2},
	urldate = {2025-07-04},
	journal = {Chemometrics and Intelligent Laboratory Systems},
	author = {Forina, Michele and Lanteri, Silvia and Casale, Monica and Cerrato Oliveros, M. Conception},
	month = jun,
	year = {2007},
	pages = {252--261},
}

@article{randic,
	title = {Mathematical chemistry illustrations: a personal view of less known results},
	volume = {57},
	issn = {0259-9791, 1572-8897},
	shorttitle = {Mathematical chemistry illustrations},
	url = {http://link.springer.com/10.1007/s10910-018-0951-0},
	doi = {10.1007/s10910-018-0951-0},
	language = {en},
	number = {1},
	urldate = {2025-07-04},
	journal = {Journal of Mathematical Chemistry},
	author = {Randić, Milan},
	month = jan,
	year = {2019},
	pages = {280--314},
}

@book{lehmann,
	address = {Cham},
	series = {Springer {Texts} in {Statistics}},
	title = {Testing {Statistical} {Hypotheses}},
	copyright = {https://www.springer.com/tdm},
	isbn = {978-3-030-70577-0 978-3-030-70578-7},
	url = {https://link.springer.com/10.1007/978-3-030-70578-7},
	language = {en},
	urldate = {2025-07-04},
	publisher = {Springer International Publishing},
	author = {Lehmann, E.L. and Romano, Joseph P.},
	year = {2022},
	doi = {10.1007/978-3-030-70578-7},
}

@incollection{bhattacharya,
	title = {Hypothesis {Testing}},
	copyright = {https://www.elsevier.com/tdm/userlicense/1.0/},
	isbn = {978-0-12-802440-9},
	url = {https://linkinghub.elsevier.com/retrieve/pii/B9780128024409000060},
	language = {en},
	urldate = {2025-07-04},
	booktitle = {Theory and {Methods} of {Statistics}},
	publisher = {Elsevier},
	author = {Bhattacharya, P.K. and Burman, Prabir},
	year = {2016},
	doi = {10.1016/B978-0-12-802440-9.00006-0},
	pages = {125--177},
}

@article{king,
	title = {Joint one-sided tests of linear regression coefficients},
	volume = {32},
	copyright = {https://www.elsevier.com/tdm/userlicense/1.0/},
	issn = {03044076},
	url = {https://linkinghub.elsevier.com/retrieve/pii/0304407686900205},
	doi = {10.1016/0304-4076(86)90020-5},
	language = {en},
	number = {3},
	urldate = {2025-07-04},
	journal = {Journal of Econometrics},
	author = {King, Maxwell L. and Smith, Murray D.},
	month = aug,
	year = {1986},
	pages = {367--383},
}

@misc{laplace,
	title = {Translation and modern interpretation of {Laplace}'s {Théorie} {Analytique} des {Probabilités}, pages 505-512, 516-520},
	url = {http://arxiv.org/abs/0907.4695},
	doi = {10.48550/arXiv.0907.4695},
	urldate = {2025-07-17},
	publisher = {arXiv},
	author = {Langou, Julien},
	month = jul,
	year = {2009},
	note = {arXiv:0907.4695 [math]},
}

@book{randic.book,
	edition = {0},
	title = {Solved and {Unsolved} {Problems} of {Structural} {Chemistry}},
	isbn = {978-0-429-18163-4},
	url = {https://www.taylorfrancis.com/books/9781498711524},
	language = {en},
	urldate = {2025-07-18},
	publisher = {CRC Press},
	author = {Rand{\'\i}{\'c}, Milan and Novi{\v{c}}, Marjana and Plav{\v{s}}i{c}, Dejan},
	month = apr,
	year = {2016},
	doi = {10.1201/b19046},
}

@misc{cross,
	title = {Treatment effects without multicollinearity? {Temporal} order and the {Gram}-{Schmidt} process in causal inference},
	shorttitle = {Treatment effects without multicollinearity?},
	url = {http://arxiv.org/abs/2402.17103},
	doi = {10.48550/arXiv.2402.17103},
	language = {en},
	urldate = {2025-07-18},
	publisher = {arXiv},
	author = {Cross, Robin M. and Buccola, Steven T.},
	month = jan,
	year = {2025},
	note = {arXiv:2402.17103 [econ]},
}

@article{farebrother,
	title = {Algorithm {AS} 79: {Gram}-{Schmidt} {Regression}},
	volume = {23},
	issn = {0035-9254},
	shorttitle = {Algorithm {AS} 79},
	url = {https://www.jstor.org/stable/2347151?origin=crossref},
	doi = {10.2307/2347151},
	language = {en},
	number = {3},
	urldate = {2025-07-18},
	journal = {Applied Statistics},
	author = {Farebrother, R. W.},
	year = {1974},
	note = {Publisher: JSTOR},
	pages = {470},
}

@article{klein,
	title = {Hierarchical orthogonalization of descriptors},
	volume = {63},
	copyright = {http://doi.wiley.com/10.1002/tdm\_license\_1.1},
	issn = {0020-7608, 1097-461X},
	url = {https://onlinelibrary.wiley.com/doi/10.1002/(SICI)1097-461X(1997)63:1<215::AID-QUA22>3.0.CO;2-9},
	doi = {10.1002/(sici)1097-461x(1997)63:1<215::aid-qua22>3.0.co;2-9},
	language = {en},
	number = {1},
	urldate = {2025-07-18},
	journal = {International Journal of Quantum Chemistry},
	author = {Klein, Douglas J. and Randić, Milan and Babić, Darko and Lucić, Bono and Nikolić, Sonja and Trinajstić, Nenad},
	year = {1997},
	note = {Publisher: Wiley},
	pages = {215--222},
}

@article{portnoy,
	title = {Linearity of {Unbiased} {Linear} {Model} {Estimators}},
	volume = {76},
	issn = {0003-1305, 1537-2731},
	url = {https://www.tandfonline.com/doi/full/10.1080/00031305.2022.2076743},
	doi = {10.1080/00031305.2022.2076743},
	language = {en},
	number = {4},
	urldate = {2025-10-02},
	journal = {The American Statistician},
	author = {Portnoy, Stephen},
	month = oct,
	year = {2022},
	pages = {372--375},
}

@article{hoerl,
	title = {Ridge {Regression}: {Biased} {Estimation} for {Nonorthogonal} {Problems}},
	volume = {12},
	issn = {0040-1706},
	shorttitle = {Ridge {Regression}},
	url = {https://www.tandfonline.com/doi/abs/10.1080/00401706.1970.10488634},
	doi = {10.1080/00401706.1970.10488634},
	number = {1},
	urldate = {2025-10-02},
	journal = {Technometrics},
	author = {Hoerl, Arthur E. and Kennard, Robert W.},
	month = feb,
	year = {1970},
	note = {Publisher: ASA Website
\_eprint: https://www.tandfonline.com/doi/pdf/10.1080/00401706.1970.10488634},
	pages = {55--67},
}

@article{perez-melo,
	title = {On {Some} {Test} {Statistics} for {Testing} the {Regression} {Coefficients} in {Presence} of {Multicollinearity}: {A} {Simulation} {Study}},
	volume = {3},
	copyright = {https://creativecommons.org/licenses/by/4.0/},
	issn = {2571-905X},
	shorttitle = {On {Some} {Test} {Statistics} for {Testing} the {Regression} {Coefficients} in {Presence} of {Multicollinearity}},
	url = {https://www.mdpi.com/2571-905X/3/1/5},
	doi = {10.3390/stats3010005},
	language = {en},
	number = {1},
	urldate = {2025-10-02},
	journal = {Stats},
	author = {Perez-Melo, Sergio and Kibria, B. M. Golam},
	month = mar,
	year = {2020},
	pages = {40--55},
}

@article{halawa,
	title = {Tests of regression coefficients under ridge regression models},
	volume = {65},
	issn = {0094-9655, 1563-5163},
	url = {http://www.tandfonline.com/doi/abs/10.1080/00949650008812006},
	doi = {10.1080/00949650008812006},
	language = {en},
	number = {1-4},
	urldate = {2025-10-02},
	journal = {Journal of Statistical Computation and Simulation},
	author = {Halawa, A.M. and El Bassiouni, M.Y.},
	month = jan,
	year = {2000},
	pages = {341--356},
}

@article{liu,
	title = {Using {Liu}-{Type} {Estimator} to {Combat} {Collinearity}},
	volume = {32},
	issn = {0361-0926},
	url = {https://doi.org/10.1081/STA-120019959},
	doi = {10.1081/STA-120019959},
	number = {5},
	urldate = {2025-10-02},
	journal = {Communications in Statistics - Theory and Methods},
	author = {Liu, Kejian},
	month = jan,
	year = {2003},
	note = {Publisher: Taylor \& Francis
\_eprint: https://doi.org/10.1081/STA-120019959},
	pages = {1009--1020},
}

@article{goldberger,
	title = {Structural {Equation} {Methods} in the {Social} {Sciences}},
	volume = {40},
	issn = {00129682},
	url = {https://www.jstor.org/stable/1913851?origin=crossref},
	doi = {10.2307/1913851},
	language = {en},
	number = {6},
	urldate = {2025-10-02},
	journal = {Econometrica},
	author = {Goldberger, Arthur S.},
	month = nov,
	year = {1972},
	pages = {979},
}

@book{pearl,
	address = {Cambridge, U.K. ; New York},
	title = {Causality},
	isbn = {978-0-521-89560-6},
	language = {English},
	publisher = {Cambridge University Press},
	author = {Pearl, Judea},
	year = {2000},
}

@article{hansen,
	title = {A {Modern} {Gauss}–{Markov} {Theorem}},
	volume = {90},
	copyright = {© 2022 The Econometric Society},
	issn = {1468-0262},
	url = {https://onlinelibrary.wiley.com/doi/abs/10.3982/ECTA19255},
	doi = {10.3982/ECTA19255},
	language = {en},
	number = {3},
	urldate = {2025-10-02},
	journal = {Econometrica},
	author = {Hansen, Bruce E.},
	year = {2022},
	note = {\_eprint: https://onlinelibrary.wiley.com/doi/pdf/10.3982/ECTA19255},
	pages = {1283--1294},
}

@misc{potscher,
	title = {A {Modern} {Gauss}-{Markov} {Theorem}? {Really}?},
	shorttitle = {A {Modern} {Gauss}-{Markov} {Theorem}?},
	url = {http://arxiv.org/abs/2203.01425},
	doi = {10.48550/arXiv.2203.01425},
	urldate = {2025-10-02},
	publisher = {arXiv},
	author = {Pötscher, Benedikt M. and Preinerstorfer, David},
	month = oct,
	year = {2023},
	note = {arXiv:2203.01425 [math]},
}

@article{zhang,
	title = {Uniformly most powerful tests under weak restrictions},
	volume = {65},
	issn = {0932-5026, 1613-9798},
	url = {https://link.springer.com/10.1007/s00362-023-01479-0},
	doi = {10.1007/s00362-023-01479-0},
	language = {en},
	number = {4},
	urldate = {2025-10-02},
	journal = {Statistical Papers},
	author = {Zhang, Jin},
	month = jun,
	year = {2024},
	pages = {2211--2220},
}

@article{thornton,
	title = {The {Demand} for, and {Impact} of, {Learning} {HIV} {Status}},
	volume = {98},
	issn = {0002-8282},
	url = {https://www.jstor.org/stable/29730154},
	number = {5},
	urldate = {2025-11-04},
	journal = {The American Economic Review},
	author = {Thornton, Rebecca L.},
	year = {2008},
	note = {Publisher: American Economic Association},
	pages = {1829--1863},
}

@misc{kenny,
	title = {Mediation},
	url = {https://davidakenny.net/cm/mediate.htm},
	urldate = {2026-03-18},
	year = {2026},
	author = {Kenny, David A.},
}

@article{orourke,
	title = {When the test of mediation is more powerful than the test of the total effect},
	volume = {47},
	issn = {1554-3528},
	url = {http://link.springer.com/10.3758/s13428-014-0481-z},
	doi = {10.3758/s13428-014-0481-z},
	language = {en},
	number = {2},
	urldate = {2026-03-18},
	journal = {Behavior Research Methods},
	author = {O’Rourke, Holly P. and MacKinnon, David P.},
	month = jun,
	year = {2015},
	pages = {424--442},
}

@article{sobel,
	title = {Asymptotic {Confidence} {Intervals} for {Indirect} {Effects} in {Structural} {Equation} {Models}},
	volume = {13},
	issn = {0081-1750},
	url = {https://www.jstor.org/stable/270723},
	doi = {10.2307/270723},
	urldate = {2026-03-19},
	journal = {Sociological Methodology},
	author = {Sobel, Michael E.},
	year = {1982},
	note = {Publisher: [American Sociological Association, Wiley, Sage Publications, Inc.]},
	pages = {290--312},
}

@incollection{rawlings,
	address = {New York, NY},
	title = {Class {Variables} in {Regression}},
	isbn = {978-0-387-22753-5},
	url = {https://doi.org/10.1007/0-387-22753-9_9},
	language = {en},
	urldate = {2026-04-15},
	booktitle = {Applied {Regression} {Analysis}: {A} {Research} {Tool}},
	publisher = {Springer},
	editor = {Rawlings, John O. and Pantula, Sastry G. and Dickey, David A.},
	year = {1998},
	doi = {10.1007/0-387-22753-9_9},
	pages = {269--323},
}

@article{faraway,
	title = {Practical {Regression} and {Anova} using {R}},
	language = {en},
	author = {Faraway, Julian J},
	year = {2002},
}

@article{marsaglia,
	title = {Ratios of {Normal} {Variables} and {Ratios} of {Sums} of {Uniform} {Variables}},
	language = {en},
	author = {Marsaglia, George},
	year = {1965},
}

@misc{dufour,
	title = {Causal mechanism and mediation analysis for macroeconomics dynamics: a bridge of {Granger} and {Sims} causality},
	shorttitle = {Causal mechanism and mediation analysis for macroeconomics dynamics},
	url = {http://arxiv.org/abs/2509.05284},
	doi = {10.48550/arXiv.2509.05284},
	language = {en},
	urldate = {2026-09-23},
	publisher = {arXiv},
	author = {Dufour, Jean-Marie and Wang, Endong},
	month = sep,
	year = {2025},
	note = {arXiv:2509.05284 [econ.EM]},
}

@article{jorda,
	title = {Estimation and {Inference} of {Impulse} {Responses} by {Local} {Projections}},
	volume = {95},
	issn = {0002-8282},
	url = {https://pubs.aeaweb.org/doi/10.1257/0002828053828518},
	doi = {10.1257/0002828053828518},
	language = {en},
	number = {1},
	urldate = {2026-09-23},
	journal = {American Economic Review},
	author = {Jord\`a, \`Oscar},
	month = feb,
	year = {2005},
	pages = {161--182},
}

@inproceedings{peters,
	  title = {Causal Inference on Time Series using Restricted Structural Equation Models},
	  booktitle = {Advances in Neural Information Processing Systems 26},
	  pages = {154-162},
	  editors = {C.J.C. Burges, L. Bottou, M. Welling, Z. Ghahramani, and K.Q. Weinberger},
	  year = {2013},
	  author = {Peters, J. and Janzing, D. and Sch{\"o}lkopf, B.},
}

@book{spirtes,
	address = {Cambridge, MA, USA},
	edition = {2},
	series = {Adaptive {Computation} and {Machine} {Learning} series},
	title = {Causation, {Prediction}, and {Search}},
	isbn = {978-0-262-52792-7},
	language = {en},
	publisher = {MIT Press},
	author = {Spirtes, Peter and Glymour, Clark and Scheines, Richard},
	editor = {Bach, Francis},
	month = jan,
	year = {2001},
}

@article{sims1980,
	title = {Macroeconomics and {Reality}},
	volume = {48},
	issn = {0012-9682},
	url = {https://www.jstor.org/stable/1912017},
	doi = {10.2307/1912017},
	number = {1},
	urldate = {2026-09-23},
	journal = {Econometrica},
	author = {Sims, Christopher A.},
	year = {1980},
	note = {Publisher: [Wiley, Econometric Society]},
	pages = {1--48},
}

@software{google,
  author       = {{Google}},
  title        = {Gemini},
  version      = {Large language model},
  year         = {2026},
  url          = {https://gemini.google.com},
  urldate      = {2026-09-23}
}

\end{document}